\documentclass{amsart}
\usepackage[utf8]{inputenc}

\title[A multi-point maximum principle for Harnack inequalities]{A multi-point maximum principle to prove the global Harnack inequality for the doubly nonlinear heat equation}
\author{Ben Andrews, Daniel Hauer and Jessica Slegers}
\date{\today}

\address{Ben Andrews, Mathematical Sciences Institute, Australian National University, Canberra, ACT, 2601, Australia }
\email{Ben.Andrews@anu.edu.au}
\address{Daniel Hauer, Brandenburg University of Technology Cottbus–Senftenberg, Platz der Deutschen Einheit 1, 03046 Cottbus, Germany; School of Mathematics and Statistics, The University of Sydney, Sydney, NSW, 2006, Australia }
\email{daniel.hauer@b-tu.de}
\address{
    Jessica Slegers, School of Mathematics and Statistics, The University of Sydney, Sydney, NSW, 2006, Australia; Brandenburg University of Technology Cottbus–\linebreak Senftenberg, Platz der Deutschen Einheit 1, 03046 Cottbus, Germany}
\email{jessica.slegers@sydney.edu.au}

\usepackage{amsmath,amssymb}
\usepackage{bbm}
\usepackage{mathtools}
\usepackage[colorlinks,citecolor=blue,urlcolor=blue]{hyperref}
\usepackage{enumitem}
\setenumerate{label=(\roman*)}

\usepackage{amsthm}

\newtheoremstyle{theorem}{}{}{\itshape}{}{\bfseries}{.}{ }{}
\theoremstyle{theorem}
\newtheorem{theorem}{Theorem}[section]

\newtheorem{lemma}[theorem]{Lemma}
\newtheorem{corollary}[theorem]{Corollary}

\newtheoremstyle{definition}{}{}{}{}{\bfseries}{.}{ }{}
\theoremstyle{definition}
\newtheorem{definition}{Definition}[section]
\newtheorem{remark}{Remark}[section]

\numberwithin{equation}{section} 

\newcommand{\R}{\mathbb{R}}

\newcommand{\dd}{\ \mathrm{d}}
\DeclareMathOperator{\divergence}{div}
\DeclareMathOperator{\trace}{tr}

\DeclareMathOperator{\sym}{Sym}

\DeclareMathOperator{\USC}{USC}
\DeclareMathOperator{\LSC}{LSC}

\thanks{}


 
\begin{document}

\begin{abstract}
This work aims to present a new proof of the global pointwise Harnack inequality satisfied by positive viscosity solutions of the doubly nonlinear heat equation on $\R^d$. Our approach is based on a multi-point maximum principle argument introduced by the authors in a recent article [Calc. Var., 2026]. Required for the proof is a new variation of the maximum principle for semicontinuous functions, which is a standard tool in the analysis of viscosity solutions. In particular, our techniques provide a new proof of the Harnack inequalities satisfied by positive solutions of the porous medium equation and the $p$-heat equation, which were derived by Auchmuty and Bao [Proc. Amer. Math. Soc., 1994] using results of Aronson and B\'enilan [C. R. Acad. Sci. Paris, 1979] and Esteban and V\'azquez [Arch. Rational Mech. Anal., 1988]. The results contained in this article have also been discussed in the third author's doctoral thesis.
\end{abstract}

\maketitle

\section{Introduction and main results}

This article is a continuation of a recent work by the authors \cite{paper1}, which introduced a new methodology based on a multi-point maximum principle argument to prove global Harnack inequalities satisfied by positive solutions of parabolic equations. One advantage of the method employed in \cite{paper1} is that the solutions were only assumed to have $C^{2,1}$ regularity, compared with previous works, which typically require solutions to be at least $C^4$ in the spatial variable (see for example, \cite{Aronson-Benilan-1979,Esteban-Vasquez-1990, LiYau1986}). Since our methodology can be performed even under low regularity assumptions, it is well-suited to establishing global Harnack inequalities for viscosity solutions of nonlinear evolution problems, which do not possess any differentiability properties in general. In this article, we demonstrate this by using the doubly nonlinear heat equation (also called Leibenson's equation \cite{Leibenson}) as a prototypical example. Since solutions of this equation do not enjoy sufficient regularity to be treated as classical solutions \cite{boegelein2023holdercontinuitygradientsolutions, dibenedetto-friedman-pheat-regularity-addendum,dibenedetto-friedman-pheat-regularity, strunk-dne}, it is appropriate to analyse solutions of this equation via a weaker notion, such as viscosity solutions. We emphasise, however, that our method can also handle more complicated nonlinear problems. 

For $m>0$, $p \geq 2$, we consider the equation 
\begin{equation}\label{eq:dne-manifold}
    \partial_t u = \Delta_{p} (u^m) \qquad \text{in $\Omega \times (0,\infty)$},
\end{equation}
on either a bounded, convex domain $\Omega \subseteq \R^d$ or the full Euclidean space $\R^d$, where \[
\Delta_p u:= \divergence(|\nabla u|^{p-2}\nabla u)
\]
is the \emph{$p$-Laplace operator}. Setting $p=2$ in equation \eqref{eq:dne-manifold} corresponds to the porous medium equation while choosing $m=1$ yields the $p$-heat equation. Although we restrict the parameter $p$ to the range $p\geq 2$, our methodology can also be applied to study the case $p<2$, for which the operator $\Delta_p$ becomes singular. However, the notion of viscosity solutions of \eqref{eq:dne-manifold} becomes more delicate to define in this singular range. Therefore, we choose to restrict ourselves to the case $p \geq 2$ and handle the case $p<2$ in a forthcoming paper.

The problem of studying the Harnack inequality satisfied by solutions of quasilinear differential equations such as \eqref{eq:dne-manifold} has been well-studied since the 1960s. Although fast progress was made in the elliptic case using the methods of Moser and De Giorgi \cite{Ladyzhenskaya-elliptic, serrin1964, trudinger-1967}, the parabolic case proved much more difficult to understand. In the latter half of the 1980s, DiBenedetto introduced a new metholodogy, referred to as the \emph{method of intrinsic scaling} (see \cite{DiBenedetto-Harnack-Book} or \cite{Urbano-intrinsic}), in order to study the Hölder regularity properties of quasilinear parabolic equations of $p$-Laplace type \cite{Chen-DiBenedetto, DiBenedetto-1986}. This approach also enabled the discovery of local Harnack inequalities for the $p$-heat equation in both the cases $p\geq 2$ \cite{DiBenedetto-1988-p-pme}, as well as for $\frac{2d}{d+1} < p < 2$ together with Kwong \cite{dibenedetto-kwong}. Interestingly, DiBenedetto and Kwong showed in these articles, that these techniques could also be applied to prove Harnack inequalities for positive solutions of the porous medium equation for $m> \frac{(d-2)_+}{d}$. Results for the doubly nonlinear heat equation \eqref{eq:dne-manifold} were also achieved by Gianazza and Vespri \cite{gianazza-vespri, vespri}, which were significant since the earlier methods of Moser had only yielded results in the specific case $m=\frac{1}{p-1}$ \cite{kinnunen-kuusi, kuusi, trudinger-1968}. Harnack inequalities for general quasilinear parabolic equations were found much later by DiBenedetto, Gianazza, and Vespri \cite{DGV-paper, DGVpless2}, again subject to the condition that $p>p_{*}:= \frac{2d}{d+1}$. The work of these three authors revealed several interesting phenomena, and we refer to their book \cite{DiBenedetto-Harnack-Book} for a detailed, self-contained account of their findings.

We note that prior to the work of DiBenedetto, the differential Harnack inequality for the porous medium equation had already been proven by Aronson and B\'enilan \cite{Aronson-Benilan-1979}, and similar estimates were later found for the $p$-heat equation and doubly nonlinear heat equation by Esteban and V\'azquez \cite{Esteban-Vazquez-1988, Esteban-Vasquez-1990}. These inequalities are analogous to the famous Li-Yau differential Harnack inequality for the linear heat equation \cite{LiYau1986} and were derived using similar techniques based on a comparison principle. However, the space-time integration of these differential Harnack inequalities to obtain pointwise global Harnack inequalities was not explicitly performed until Auchmuty and Bao's article \cite{Auchmuty-Bao-1994}, which credits Li and Yau's article \cite{LiYau1986} as a main inspiration for the integration step. A differential Harnack inequality was also made available by Crandall and Pierre \cite[Lemma 2]{Crandall-Pierre} for the generalised porous medium equation $\partial_t u = \Delta (\varphi(u))$, where $\varphi:\R \to \R$ is a non-decreasing function. 

Harnack inequalities for nonlinear heat equations have also been studied extensively in the setting of Riemannian manifolds, with several published works being concerned in particular with equations of porous medium type \cite{Bailesteanu, cao-zhu, Huang-Huang-Li, Lu-Ni-Vazquez-Villani, Ma-Li,  Ma-Zhao-Song, Qiu, Vazquez-PME-book, Wang-2017, wang-xie-zhou, Yau-1994}, as well as the $p$-heat equation \cite{Kotschwar-Ni} and the doubly nonlinear heat equation \cite{azami, chen-xiong-dne, Wang-Chen-doubly-nonlinear, wang-wang}. The articles mentioned so far typically consider classical or weak solutions, although some further works also exist, which treat viscosity solutions of fully nonlinear equations directly \cite{caffarelli, gripenberg, kim-lee, Kurkinen-Siltakoski, parviainen}. 

The main results of this article can already be found in the third author's doctoral thesis \cite{my-phd}. During the preparation of this article, our multi-point maximum principle technique was further adapted by Jeffres and Li \cite{jeffres-li} to treat viscosity solutions of parabolic equations associated to various other nonlinear operators, including the normalised $p$-Laplace operator and weighted $p$-Laplace operator. Their results also apply the methodology to the settings of Riemannian manifolds and closed smooth metric measure spaces. What remains new in our article is the treatment of the doubly nonlinear heat equation in Euclidean spaces, although the proof of Theorem 1.6 in \cite{jeffres-li} can be easily adapted to extend our results to these more general settings.

\begin{theorem}\label{thm:dne-viscosity}
Let $m>0$, $p\geq 2$ and define $\lambda:= m - \frac{1}{p-1}$. Set $S:= \{(t,s) \in \R^2 \mid 0 < s < t \}$ and let $\Omega \subset \R^d$ be a bounded, convex domain with smooth boundary $\partial \Omega$. Suppose $\omega: \overline{\Omega} \times \overline{\Omega} \times S \rightarrow \R$ is a continuous function, which is twice differentiable in the first and second arguments and differentiable in the third and fourth arguments. Suppose that $\omega(x,y;t,s) \rightarrow \infty$ whenever $t \rightarrow s^+$ for $x\ne y$ and $\lim_{t\to s^+} \omega(x,x;t,s) \geq 0$ for all $x \in \overline{\Omega}$. In addition, suppose there exists a strictly increasing, differentiable function $\beta:[0,\infty) \to [0,\infty)$ with $\beta(0)=0$ and a function $A:[0,\infty) \to [0,\infty)$ such that $A^2\beta^{\lambda}$ is increasing and $\omega$ satisfies
\begin{align}
\partial_{t} \omega + \beta^{-\lambda(p-1)}(t)|\nabla_{x}\omega|^{p} &\geq 0\qquad \text{ in $\overline{\Omega} \times \overline{\Omega} \times S$},\label{eq:dne-euclidean-omega-ineq-1-visc} \\
\partial_{s} \omega - \beta^{-\lambda(p-1)}(s)|\nabla_{y}\omega|^{p} &\geq 0\qquad \text{ in $\overline{\Omega} \times \overline{\Omega} \times S$},\label{eq:dne-euclidean-omega-ineq-2-visc}
\end{align} 
\begin{equation}\label{eq:dne-euclidean-omega-ineq-3-visc}
\begin{split}
    A^{2}(t) \beta^{-\lambda(p-2)}(t)\Delta_{p,x}\omega
+ A^{2}(s) \beta^{-\lambda(p-2)}(s)\Delta_{p,y}\omega& 
\\+ 2A(t)A(s)\beta^{-\frac{\lambda(p-2)}{2}}(t)\beta^{-\frac{\lambda(p-2)}{2}}(s)\Delta_{p,xy} \omega& \\ \leq A^{2}(t)\beta^{\lambda-1}(t) \beta'(t)-& A^{2}(s) \beta^{\lambda-1}(s) \beta'(s)
\end{split}
\end{equation}
in $\overline{\Omega} \times \overline{\Omega} \times S$, where \[
\begin{split}
    \Delta_{p,xy} \omega:= |\nabla_{x}\omega|&^{\frac{p-4}{2}}|\nabla_{y}\omega|^{\frac{p-4}{2}}\times \\&\left(|\nabla_{x}\omega||\nabla_{y}\omega|\sum_{i=1}^d \frac{\partial^2 \omega}{\partial x_i\partial y_i} - (p-2)\sum_{i,j=1}^d \frac{\partial \omega}{\partial x_i} \frac{\partial^2 \omega}{\partial x_i\partial y_j} \frac{\partial \omega}{\partial y_j}\right).
\end{split}
\] Finally, assume that 
\begin{equation}\label{eq:V-convexity-compact-dne-visc}
    \begin{split}
    \nabla_x \omega \cdot \nu(x) \geq 0, \qquad \text{for all $x \in \partial \Omega$, $y \in \overline{\Omega}$, and $0<s<t$,}\\
    \nabla_y \omega \cdot \nu(y) \geq 0, \qquad \text{for all $x \in \overline{\Omega}$, $y \in \partial \Omega$, and $0<s<t$.}
\end{split}
\end{equation}
Then every positive viscosity solution $u$ of
\begin{equation}\label{eq:dne-euclidean-neumann-visc}
\begin{cases}
        \partial_t u = \Delta_{p} (u^m)  &\text{in $\Omega \times (0,\infty)$} \\
        |\nabla (u^m)|^{p-2}\nabla (u^m) \cdot \nu = 0 \quad &\text{on $\partial \Omega \times (0,\infty)$}
    \end{cases}
\end{equation}
satisfies the Harnack inequality
\begin{equation}\label{eq:harnack-dne-lam>0-visc}
    \psi(\beta(t)u(x,t)) \geq \psi(\beta(s)u(y,s)) - \omega(x,y;t,s)
\end{equation}
for all $x,y \in \overline{\Omega}$ and $0<s<t$, where 
\[
\psi(u):= \begin{cases}
   \dfrac{m}{\lambda}u^{\lambda} & \text{if $\lambda \ne 0$,}\\[12pt] 
    \dfrac{1}{p-1}\log u & \text{if $\lambda=0$.}
\end{cases}
\]
\end{theorem}

As a special case of Theorem \ref{thm:dne-viscosity}, we obtain the following result for the linear heat equation, which extends the result previously obtained in \cite{paper1} to viscosity solutions.

\begin{theorem}\label{thm:Euclidean-Schroedinger-full-visc}
Let $\Omega \subseteq \R^d$ be a bounded, convex domain with smooth boundary $\partial \Omega$ and let $V \in C^2(\Omega)$ be bounded from below.
Let $\omega: \overline{\Omega} \times \overline{\Omega} \times S \rightarrow\R$ be a continuous function, which is twice differentiable in the first and second arguments and differentiable in the third and fourth arguments. Suppose that $\omega(x,y;t,s) \rightarrow \infty$ whenever $t \rightarrow s^+$ for $x\ne y$ and $\lim_{t\to s^+} \omega(x,x;t,s) \geq 0$ for all $x \in \overline{\Omega}$. In addition, assume that $\omega$ satisfies
\begin{align}
    \partial_t \omega + |\nabla_x\omega|^2 &\geq V(x) \ \ \qquad \text{ in $\overline{\Omega} \times \overline{\Omega} \times S$}, \label{eq:schroe-omega-ineq-1-visc}\\
    \partial_s \omega - |\nabla_y\omega|^2 &\geq -V(y) \qquad \text{ in $\overline{\Omega} \times \overline{\Omega} \times S$,}\label{eq:schroe-omega-ineq-2-visc}
\end{align}
and there exist strictly increasing, differentiable functions $\beta:[0,\infty) \to [0,\infty)$ satisfying $\beta(0)=0$ and $A:[0,\infty) \to [0,\infty)$ such that
\begin{equation}
    A(t)^2\Delta_x\omega +A(s)^2\Delta_y \omega + 2A(t)A(s) \Delta _{xy} \omega \leq A(t)^2 \frac{\beta'(t)}{\beta(t)}-A(s)^2 \frac{\beta'(s)}{\beta(s)}\label{eq:schroe-omega-ineq-3-visc}
\end{equation}
holds in $\overline{\Omega} \times \overline{\Omega} \times S$.
Finally, assume that 
\begin{equation}\label{eq:V-convexity-compact-visc}
    \begin{split}
    \nabla_x \omega \cdot \nu(x) \geq 0, \qquad \text{for all $x \in \partial \Omega$, $y \in \overline{\Omega}$, and $0<s<t$,}\\
    \nabla_y \omega \cdot \nu(y) \geq 0, \qquad \text{for all $x \in \overline{\Omega}$, $y \in \partial \Omega$, and $0<s<t$.}
\end{split}
\end{equation} Then, every positive viscosity solution $u$ of 
\begin{equation}\tag{\ref{2025-10-21: 1}}
   \begin{cases}
   \begin{alignedat}{2}
       \partial_t u &= \Delta u && \qquad \text{in $\Omega \times (0,T)$,} \\
        \nabla u \cdot \nu &= 0 && \qquad \text{on $\partial \Omega \times (0,T)$}
   \end{alignedat}
\end{cases} 
\end{equation} satisfies
\begin{equation}\label{eq:heat-harnack-omega-visc}
u(x,t) \geq u(y,s) \frac{\beta(s)}{\beta(t)} e^{-\omega(x,y;t,s)}
\end{equation}
for all $x,y \in \Omega$ and $0 <s<t$.
\end{theorem}

The strategy to prove Theorem \ref{thm:dne-viscosity} and Theorem \ref{thm:Euclidean-Schroedinger-full-visc} is very similar to the one implemented in \cite{paper1}. Seeking to show that all positive solutions $u$ of \eqref{eq:dne-manifold} satisfy \eqref{eq:harnack-dne-lam>0-visc}, we first observe that after making the transformation $v=\psi (u)$, proving \eqref{eq:harnack-dne-lam>0-visc} is equivalent to showing \[
Z(x,y;t,s):= v(x,t) - v(y,s) +  \omega(x,y;t,s) + \frac{\varepsilon}{2}(t-s)^2 + \varepsilon >0
\]
for all $x,y \in \overline{\Omega}$, $0 < s < t$ and $\varepsilon >0$. We demonstrate that $Z>0$ ``initially'', that is, when $t$ and $s$ are both sufficiently close to zero. Then, we suppose there is a point $P_0 =(x_0,y_0;t_0,s_0)\in \overline{\Omega} \times \overline{\Omega} \times S$ such that the function $Z$ touches zero for the ``first time'' at $P_0$ (see Definition \ref{def:first-touching-point}) and use conditions on the derivatives of $Z$ at $P_0$ to arrive at a contradiction, implying that $Z$ must instead remain positive for all times $0<s<t$.

It was demonstrated in \cite{boegelein-strunk} (see also \cite{JLM-equiv, siltakoski} for the case $m=1$), that all bounded weak solutions of the doubly nonlinear heat equation are viscosity solutions. Therefore, the statements of Theorem \ref{thm:dne-viscosity} and Theorem \ref{thm:Euclidean-Schroedinger-full-visc} remain true if we replace viscosity solutions with weak solutions. We also leverage this fact to obtain the following result in the full-space setting.

\begin{theorem}\label{thm:dne-full-euclidean}
Let $m>0$, $p\geq 2$. Suppose $\omega: \R^d \times \R^d \times S \rightarrow \R$ is a continuous function, which is twice differentiable in the first and second arguments and differentiable in the third and fourth arguments. Suppose that $\omega(x,y;t,s) \rightarrow \infty$ whenever $t \rightarrow s^+$ for $x\ne y$ and $\lim_{t\to s^+} \omega(x,x;t,s) \geq 0$ for all $x \in \R^d$. In addition, suppose there exists a strictly increasing, differentiable function $\beta:[0,\infty) \to [0,\infty)$ with $\beta(0)=0$ and a function $A:[0,\infty) \to [0,\infty)$ such that $A^2\beta^{\lambda}$ is increasing and $\omega$ satisfies \eqref{eq:dne-euclidean-omega-ineq-1-visc}--\eqref{eq:dne-euclidean-omega-ineq-3-visc}
in $\R^d \times \R^d \times S$. Then every positive weak solution $u$ of
    \begin{equation}\label{eq:dne-full-space}
        \partial_t u = \Delta_{p} (u^m) \qquad \text{in $\R^d \times (0,\infty)$}
    \end{equation}
satisfies the Harnack inequality \eqref{eq:harnack-dne-lam>0-visc}.
\end{theorem}
Theorem \ref{thm:dne-full-euclidean} can be proven by applying standard arguments from semigroup theory to approximate a positive weak solution $u$ of \eqref{eq:dne-full-space} by a sequence $(u_n)_{n \geq 1}$ of weak solutions to the homogenenous Neumann problem \eqref{eq:dne-euclidean-neumann-visc} on balls of increasing radius. We give a further explanation of this argument in the appendix. Then, since each of the weak solutions $u_n$ is also a viscosity solution, the Harnack inequality \eqref{eq:harnack-dne-lam>0-visc} holds for all $n \geq 1$ and we may pass to the limit as $n \to \infty$ to verify that $u$ satisfies \eqref{eq:harnack-dne-lam>0-visc} as well. 

Before explaining our proof of the Harnack inequality, we first present in Section \ref{sec:prelim} the necessary definitions and technical tools required in the proof, with an emphasis on a new variation of the maximum principle for semicontinuous functions. We begin Section \ref{sec:harnack-viscosity} with the proof of Theorem \ref{thm:Euclidean-Schroedinger-full-visc} to demonstrate our approach in the simpler linear case. We then prove Theorem \ref{thm:dne-viscosity} for viscosity solutions of the doubly nonlinear heat equation. 

\section{Preliminaries}\label{sec:prelim}

\subsection{Viscosity solutions of parabolic equations}

\begin{definition}[Parabolic semijets of semicontinuous functions]\label{def:par-semijets}
    Let $\Omega \subseteq \R^d$, $I_i \subset \R$ be an interval for $i=1,\ldots,k$, and $S:=I_1 \times \ldots \times  I_k \subseteq \R^k$.
    \begin{enumerate}
        \item For a function $u\in \USC(\Omega \times S)$, $x_0 \in \Omega$, $t_0 \in S$, we define the \emph{parabolic superjet} $\mathcal{P}^{2,+}_{\Omega \times S}u(x_0,t_0)$ as the set of $(a,p,X) \in \R^k \times \R^d \times \sym_{d}$ such that 
        \[
        \begin{aligned}u(x,t)  &\leq u(x_{0},t_0)+ \langle a ,t-t_0 \rangle + \langle p ,x-x_{0}\rangle \\ & \hspace{0.5cm}+ \frac{1}{2} \langle X(x-x_{0}),(x-x_{0})\rangle + o(|t_0 -t|_1 + |x-x_{0}|^{2})\end{aligned}
        \]
        as $x \to x_{0}$ in $\Omega$ and $t_i \to t_{0i} ^{-}$ for $i=1,\ldots,k$.
        \item For a function $u\in \LSC(\Omega \times S)$, $x_0 \in \Omega$, $t_0 \in S$, we define the \emph{parabolic subjet} $\mathcal{P}^{2,-}_{\Omega \times S} u(x_0,t_0)$ as the set of $(a,p,X) \in \R^k \times \R^d \times \sym_{d}$ such that 
        \[
        \begin{aligned}u(x,t)  &\geq u(x_{0},t_0)+ \langle a ,t-t_0 \rangle + \langle p ,x-x_{0}\rangle \\ & \hspace{0.5cm}+ \frac{1}{2} \langle X(x-x_{0}),(x-x_{0})\rangle + o(|t_0 -t|_1 + |x-x_{0}|^{2})\end{aligned}
        \]
        as $x \to x_{0}$ in $\Omega$ and $t_i \to t_{0i} ^{-}$ for $i=1,\ldots,k$.
    \end{enumerate}
\end{definition}  

\begin{remark}
    We immediately mention that our definition of $\mathcal{P}^{2,\pm}u(x_0,t_0)$ does not coincide with the standard definition encountered in literature, which may be found, for example, in \cite{Users-guide}. 
    \begin{enumerate}
        \item We allow for the possibility that the function $u$ may depend on multiple time variables, as is the case for the functions $\omega$ and $Z$ utilised in our proofs of the Harnack inequality.
        \item In the typical definition in literature (for $k=1$), the inequalities are understood in the two-sided limit $t \to t_0$. However, as commented by Juutinen \cite{Juutinen-2001}, one often proceeds in the parabolic setting with the intuition that the behaviour of a solution to a parabolic equation at time $t_0$ is determined only by information about the solution at times before $t_0$ and not after. Therefore, in our definition, we only consider $t_i<t_{0i}$ in the one-sided limits $t_i \to t_{0i}^{-}$.
        \item In Definition \ref{def:par-semijets}, $|\cdot|_{1}$ denotes the $1$-norm and \[
        |t_0-t|_1 = \sum_{i=1}^{k}|t_{0i}-t_i| = \sum_{i=1}^{k}(t_{0i}-t_i),
        \]
        since we always take $t_i<t_{0i}$.
        \item As is the case for $J^{2,\pm}_{\Omega}u(x_0)$, the domain of the function $u:\Omega \times S \to \R$ is often not relevant when considering $\mathcal{P}^{2,\pm}_{\Omega \times S} u(x_0,t_0)$, since the definition is equivalent when one replaces $\Omega \times S$ with any other domain for which $(x_0,t_0)$ is an interior point. In addition, if $I_i = (a,b]$ for some $i$ and we consider the parabolic jets of $u$ at a point $(x_0,t_0)$ such that $t_{0i}=b$, then since we are only concerned with the behaviour of $u$ to the left of $t_{0i}=b$, we have that $\mathcal{P}^{2,\pm}_{\Omega \times S} u(x_0,t_0) = \mathcal{P}^{2,\pm}_{\Omega \times S'} u(x_0,t_0)$, where $S'$ is constructed by replacing the interval $I_i$ in the definition of $S$ by another interval $I_i'$ such that $I_i \subset I_i'$ and $b$ is an interior point of $I_i'$.
    \end{enumerate} 
\end{remark}

We also define the closures of the parabolic semijets.
\begin{definition}
\begin{enumerate}
\item[]
    \item For given a function $u \in \USC(\Omega \times S)$, $x_0 \in \Omega$, $t_0 \in S$, we define $\overline{\mathcal{P}}^{2,+}u(x_{0},t_0)$ to be the set of all triples $(a,p,X) \in \R^k \times \R^d \times \sym_{d}$, for which there exist sequences 
    \begin{align*}
    &(a_n, p_n,X_n)_{n \geq 1} \subseteq \R^k \times \R^d \times \sym_{d} \quad \text{and} \\&(x_n,t_n)_{n \geq 1} \subseteq \Omega \times S
    \end{align*} 
    such that $(a_n, p_n,X_n) \in \mathcal{P}^{2,+}u(x_n,t_n)$ and 
    \begin{align*}
       &(a_n, p_n,X_n) \to (a, p,X), \quad 
       x_n \to x_0, \quad t_{ni} \to t_{0i}^{-}, \quad \text{and}  \\
       &u(x_n,t_n) \to u(x_0,t_0) 
    \end{align*}
    as $n \to \infty$.
    \item For given a function $u \in \LSC(\Omega \times S)$, $x_0 \in \Omega$, $t_0 \in S$, we define $\overline{\mathcal{P}}^{2,-}u(x_{0},t_0)$ to be the set of all triples $(a,p,X) \in \R^k \times \R^d \times \sym_{d}$, for which there exist sequences \begin{align*}
    &(a_n, p_n,X_n)_{n \geq 1} \subseteq \R^k \times \R^d \times \sym_{d} \quad \text{and} \\&(x_n,t_n)_{n \geq 1} \subseteq \Omega \times S
    \end{align*}  such that $(a_n, p_n,X_n) \in \mathcal{P}^{2,-}u(x_n,t_n)$ and 
    \begin{align*}
       &(a_n, p_n,X_n) \to (a, p,X), \quad 
       x_n \to x_0, \quad t_{ni} \to t_{0i}^{-}, \quad \text{and}  \\
       &u(x_n,t_n) \to u(x_0,t_0) 
    \end{align*}
    as $n \to \infty$.
\end{enumerate}
\end{definition}

Given the previous definitions, we define viscosity solutions of the Neumann problem 
\begin{equation}\label{2025-10-27: 
1}
\begin{cases}
        \partial_t u = \Delta_{p} (u^m)  &\text{in $\Omega \times (0,T)$} \\
        |\nabla (u^m)|^{p-2}\nabla (u^m) \cdot \nu = 0 \quad &\text{on $\partial \Omega \times (0,T)$}.
    \end{cases}
\end{equation}
in the following manner.

\begin{definition}\label{def:dne-viscosity}
We call a function 
\begin{enumerate}
    \item $u \in \USC(\overline{\Omega} \times (0,T))$ a \emph{viscosity subsolution of \eqref{2025-10-27: 1}} if 
   \[
   \begin{split}
       a &- m^{p-1}u(x_{0},t_{0})^{(m-1)(p-1)}|\xi|^{p-4}\\&\hspace{1cm}\Big((p-1)(m-1)u(x_{0},t_{0})^{-1}|\xi|^{4} + |\xi|^{2}\trace X + (p-2)\xi^{\top}X \xi\Big) \leq 0
   \end{split}   
    \]
for all $(a,\xi,X) \in \overline{\mathcal{P}}^{2,+}_{\overline{\Omega} \times (0,T)}u(x_0,t_0)$ and $(x_0,t_0) \in \Omega \times (0,T)$, and 
\[
\min \left\{ \begin{matrix}&a - m^{p-1}u(x_{0},t_{0})^{(m-1)(p-1)}|\xi|^{p-4}\\&\hspace{1cm}\Big((p-1)(m-1)u(x_{0},t_{0})^{-1}|\xi|^{4} + |\xi|^{2}\trace X + (p-2)\xi^{\top}X \xi\Big),\\ &m^{p-1}u(x_{0},t_{0})^{(m-1)(p-1)}|\xi|^{p-2}\xi \cdot \nu (x_0)\end{matrix}\right\} \leq 0
\]
for all $(a,\xi,X) \in \overline{\mathcal{P}}^{2,+}_{\overline{\Omega} \times (0,T)}u(x_0,t_0)$ and $(x_0,t_0) \in \partial \Omega \times (0,T)$;
\item $u \in \LSC(\overline{\Omega} \times (0,T))$ a \emph{viscosity supersolution of \eqref{2025-10-27: 1}} if 
   \[
   \begin{split}
       a &- m^{p-1}u(x_{0},t_{0})^{(m-1)(p-1)}|\xi|^{p-4}\\&\hspace{1cm}\Big((p-1)(m-1)u(x_{0},t_{0})^{-1}|\xi|^{4} + |\xi|^{2}\trace X + (p-2)\xi^{\top}X \xi\Big) \geq 0
   \end{split}   
    \]
for all $(a,\xi,X) \in \overline{\mathcal{P}}^{2,-}_{\overline{\Omega} \times (0,T)}u(x_0,t_0)$ and $(x_0,t_0) \in \Omega \times (0,T)$, and 
\[
\max \left\{ \begin{matrix}&a - m^{p-1}u(x_{0},t_{0})^{(m-1)(p-1)}|\xi|^{p-4}\\&\hspace{1cm}\Big((p-1)(m-1)u(x_{0},t_{0})^{-1}|\xi|^{4} + |\xi|^{2}\trace X + (p-2)\xi^{\top}X \xi\Big),\\ &m^{p-1}u(x_{0},t_{0})^{(m-1)(p-1)}|\xi|^{p-2}\xi \cdot \nu (x_0)\end{matrix}\right\} \geq 0
\]
for all $(a,\xi,X) \in \overline{\mathcal{P}}^{2,-}_{\overline{\Omega} \times (0,T)}u(x_0,t_0)$ and $(x_0,t_0) \in \partial \Omega \times (0,T)$;
\item  $u \in C(\overline{\Omega} \times (0,T))$ a \emph{viscosity solution of \eqref{2025-10-27: 1}} if $u$ is both a viscosity subsolution and viscosity supersolution of \eqref{2025-10-27: 1}.
\end{enumerate}
\end{definition}

We also study the auxiliary problem obtained by applying the transformation $v=\psi(u)$ to \eqref{2025-10-27: 1}. In the case $\lambda \ne 0$, we set $v= \frac{m}{\lambda} u^{\lambda}$ and use the following definition.

\begin{definition}\label{def:dne-viscosity-transformed}
For the problem
\begin{equation}\label{2025-10-27: 2}
     \begin{cases}
        \partial_{t}v = \lambda v \Delta_p v+|\nabla v|^{p}\quad  &\text{in $\Omega \times (0,T)$} \\
        |\nabla v|^{p-2}\nabla v \cdot \nu = 0 \quad &\text{on $\partial \Omega \times (0,T)$}
    \end{cases}
\end{equation}
We call a function
    \begin{enumerate}
        \item 
$v \in \USC(\overline{\Omega} \times (0,T))$ a \emph{viscosity subsolution of \eqref{2025-10-27: 2}}
if
\[
a - \lambda v(x_{0},t_{0})|\xi|^{p-4}(|\xi|^{2}\trace X + (p-2)\xi^{\top}X \xi) - |\xi|^{p}\leq 0
\]
for all $(a,\xi,X) \in \overline{\mathcal{P}}^{2,+}_{\overline{\Omega} \times (0,T)}v(x_0,t_0)$ and $(x_0,t_0) \in \Omega \times (0,T)$ and 
\[
\min \left\{ \begin{matrix}&a - \lambda v(x_{0},t_{0})|\xi|^{p-4}(|\xi|^{2}\trace X + (p-2)\xi^{\top}X \xi) - |\xi|^{p},\\ &|\xi|^{p-2}\xi \cdot \nu (x_0)\end{matrix}\right\} \leq 0
\]
for all $(a,\xi,X) \in \overline{\mathcal{P}}^{2,+}_{\overline{\Omega} \times (0,T)}u(x_0,t_0)$ and $(x_0,t_0) \in \partial \Omega \times (0,T)$;
\item 
$v \in \LSC(\overline{\Omega} \times (0,T))$ a \emph{viscosity supersolution of \eqref{2025-10-27: 2}} if
\[
        a - \lambda v(x_{0},t_{0})|\xi|^{p-4}(|\xi|^{2}\trace X + (p-2)\xi^{\top}X \xi) - |\xi|^{p}\geq 0  
\]
for all $(a,\xi,X) \in \overline{\mathcal{P}}^{2,-}_{\overline{\Omega} \times (0,T)}v(x_0,t_0)$ and $(x_0,t_0) \in \Omega \times (0,T)$ and 
\[
 \max \left\{ \begin{matrix}&a - \lambda v(x_{0},t_{0})|\xi|^{p-4}(|\xi|^{2}\trace X + (p-2)\xi^{\top}X \xi) - |\xi|^{p},\\ &|\xi|^{p-2}\xi \cdot \nu (x_0)\end{matrix}\right\} \geq 0   
\]
for all $(a,\xi,X) \in \overline{\mathcal{P}}^{2,-}_{\overline{\Omega} \times (0,T)}u(x_0,t_0)$ and $(x_0,t_0) \in \partial \Omega \times (0,T)$;
\item $v \in C(\overline{\Omega} \times (0,T))$ a \emph{viscosity solution of \eqref{2025-10-27: 2}} if $v$ is both a viscosity subsolution and viscosity supersolution of \eqref{2025-10-27: 1}.
    \end{enumerate}
\end{definition}

Likewise, in the case $\lambda=0$, we set $v= \frac{1}{p-1} \log u$ and use the following definition.

\begin{definition}\label{def:dne-viscosity-transformed-trudinger}
For the problem
\begin{equation}\label{2025-10-27: new}
     \begin{cases}
        \partial_{t}v = \frac{1}{p-1}\Delta_p v+|\nabla v|^{p}\quad  &\text{in $\Omega \times (0,T)$,} \\
        |\nabla v|^{p-2}\nabla v \cdot \nu = 0 \quad &\text{on $\partial \Omega \times (0,T)$,}
    \end{cases}
\end{equation}
we call a function
    \begin{enumerate}
        \item 
$v \in \USC(\overline{\Omega} \times (0,T))$ a \emph{viscosity subsolution of \eqref{2025-10-27: new}} if
\[
a - \frac{1}{p-1}|\xi|^{p-4}(|\xi|^{2}\trace X + (p-2)\xi^{\top}X \xi) - |\xi|^{p}\leq 0
\]
for all $(a,\xi,X) \in \overline{\mathcal{P}}^{2,+}_{\overline{\Omega} \times (0,T)}v(x_0,t_0)$ and $(x_0,t_0) \in \Omega \times (0,T)$ and 
\[
\min \left\{ \begin{matrix}&a - \frac{1}{p-1}|\xi|^{p-4}(|\xi|^{2}\trace X + (p-2)\xi^{\top}X \xi) - |\xi|^{p},\\ &|\xi|^{p-2}\xi \cdot \nu (x_0)\end{matrix}\right\} \leq 0
\]
for all $(a,\xi,X) \in \overline{\mathcal{P}}^{2,+}_{\overline{\Omega} \times (0,T)}u(x_0,t_0)$ and $(x_0,t_0) \in \partial \Omega \times (0,T)$;
\item 
$v \in \LSC(\overline{\Omega} \times (0,T))$ a \emph{viscosity supersolution of \eqref{2025-10-27: new}} if
\[
        a - \frac{1}{p-1}|\xi|^{p-4}(|\xi|^{2}\trace X + (p-2)\xi^{\top}X \xi) - |\xi|^{p}\geq 0  
\]
for all $(a,\xi,X) \in \overline{\mathcal{P}}^{2,-}_{\overline{\Omega} \times (0,T)}v(x_0,t_0)$ and $(x_0,t_0) \in \Omega \times (0,T)$ and 
\[
 \max \left\{ \begin{matrix}&a - \frac{1}{p-1}|\xi|^{p-4}(|\xi|^{2}\trace X + (p-2)\xi^{\top}X \xi) - |\xi|^{p},\\ &|\xi|^{p-2}\xi \cdot \nu (x_0)\end{matrix}\right\} \geq 0   
\]
for all $(a,\xi,X) \in \overline{\mathcal{P}}^{2,-}_{\overline{\Omega} \times (0,T)}u(x_0,t_0)$ and $(x_0,t_0) \in \partial \Omega \times (0,T)$;
\item $v \in C(\overline{\Omega} \times (0,T))$ a \emph{viscosity solution of \eqref{2025-10-27: new}} if $v$ is both a viscosity subsolution and viscosity supersolution of \eqref{2025-10-27: new}.
    \end{enumerate}
\end{definition}

By applying the definitions of viscosity solutions and the parabolic semijets, one can show that a positive function $u \in C(\Omega \times(0,T))$ is a viscosity solution of \eqref{2025-10-27: 1} if and only if $v:=\frac{m}{\lambda}u^{\lambda}$ is a viscosity solution of \eqref{2025-10-27: 2} if $\lambda \ne 0$, or if $v:=\frac{1}{p-1}u^{\lambda}$ is a viscosity solution of \eqref{2025-10-27: new} if $\lambda \ne 0$.

\subsection{The parabolic theorem on sums}
The most important technical tool used in our proof is the maximum principle for semicontinuous functions. Although a parabolic version of this theorem is already available, for example, in Section 3 of \cite{Crandall-Ishii-maximum-semicontinuous}, it is not suitable for our purposes, since we always work with functions depending on different time variables and our definition of the parabolic semijets deviates from the one used in the existing results. Therefore, we devote the remainder of this section to proving a modified version of the parabolic theorem on sums, which is tailored to our needs. The proof is a straightforward application of the standard theorem on sums found in \cite{Crandall-Ishii-maximum-semicontinuous}. 

In order to understand the formulation of our statement, we first need to make precise the notion of a function touching a value for the first time.

\begin{definition}[First touching point]\label{def:first-touching-point}
    Let $\Omega \subseteq \R^d$, $Z: \Omega \times \Omega \times (0,\infty) \times (0,\infty) \to\R$, $c \in \R$ and $P_0:= (x_0,y_0;t_0,s_0) \in \Omega \times \Omega \times (0, \infty) \times (0, \infty)$. We say that $Z$ \emph{touches $c \in \R$ from above for the first time at $P_0$} if 
    \begin{equation}\label{eq:first-touching}
    Z(P_0) = c \quad \text{and} \quad Z(P) >c
    \end{equation}
    for all $P = (x,y;t,s) \in \Omega \times \Omega \times (0, \infty) \times (0, \infty)$ such that \[
    |(t,s)|< |(t_0, s_0)|,
    \]
    where $| \cdot |$ denotes the Euclidean norm on $\R^2$.
\end{definition}

By reversing the inequality in \eqref{eq:first-touching}, we also understand what is meant by a function $Z$ \emph{touching $c \in \R$ from below for the first time}. 

\begin{remark}
    We essentially consider pairs of times $(t,s) \in (0,\infty) \times (0, \infty)$ as being partially ordered with respect to the Euclidean norm, that is, we consider $(t,s)$ as occurring before $(t_0,s_0)$ if and only if $|(t,s)| < |(t_0,s_0)|$. Through this idea, one may also generalise Definition \ref{def:first-touching-point} for functions of any finite number of spatial or time variables. However, Definition \ref{def:first-touching-point} only compares the function values of $Z$ at $(t_0,s_0)$ with those at times before $(t_0,s_0)$. In particular, the definition makes no assumption about what occurs at other times $(t,s)$ for which $|(t,s)| = |(t_0,s_0)|$. Therefore, if $Z$ touches a value $c \in \R$ for the first time at a point $P_0 = (x_0,y_0;t_0,s_0)$, we do not exclude the possibility that there is another point $\hat{P}_0 = (\hat{x}_0,\hat{y}_0;\hat{t}_0,\hat{s}_0)$ with $|(\hat{t}_0,\hat{s}_0)| = |(t_0,s_0)|$, such that $Z$ also touches $c$ for the first time at $P_0$, but $(\hat{t}_0,\hat{s}_0) \ne (t_0,s_0)$.
\end{remark}

We first state and prove a new version of the parabolic theorem on sums adapted to our definitions.

\begin{theorem}[Parabolic theorem on sums]\label{thm:para-thm-sums}
For $i=1,\ldots,k$, let $\Omega_{i} \subseteq \R^{d_{i}}$ be locally compact and set $d:=d_{1}+\ldots+d_{k}$. Let $u_{i} \in \USC(\overline{\Omega_{i}}\times[S_{i},T_{i}])$ and consider the function
\[
u(x,t) := u_{1}(x_{1},t_{1}) + \ldots + u_{k}(x_{k},t_{k})
\]
defined for all $x:=(x_{1},\ldots,x_{k})^{\top}$, $t:=(t_{1},\ldots,t_{k})^{\top}$ with $x_{i} \in \Omega_{i}$ and $t_{i} \in (S_{i},T_{i}]$.
Suppose that for $a:=(a_{1},\ldots,a_{k})^{\top} \in \R^{k}$, $p:= (p_{1},\ldots,p_{k})^{\top} \in \R^{d}$, $X \in \sym_{d}$, and $(x,t) \in \Omega_{1}\times \ldots \times\Omega _{k} \times (S_{1},T_{1}] \times\ldots\times (S_{k},T_{k}]$, one has that
\[
(a,p,X) \in \mathcal{P}^{2,+}u(x,t).
\]
Then for all $\kappa >0$, there are matrices $X_{i} \in \sym_{d_{i}}$ such that for all $i=1,\ldots,k$,
\[
(a_{i},p_{i},X_{i}) \in \overline{\mathcal{P}}^{2,+}u_{i}(x_{i},t_{i})
\]
and
\[
-\left( \frac{1}{\kappa} + \|X\| \right)I_{d} \leq \begin{pmatrix}
X_{1} & \cdots & 0 \\ 
\vdots & \ddots & \vdots\\
0 & \cdots & X_{k}
\end{pmatrix} \leq X + \kappa X^{2}.
\]
\end{theorem}

\begin{proof}
    
By the definition of $(a,p,X) \in \mathcal{P}^{2,+}u(x,t)$, one has
\[
\begin{split}
u(y,s) &\leq u(x,t) + \langle a ,s-t\rangle + \langle p ,y-x\rangle \\
&\hspace{1cm}+ \frac{1}{2}\langle X(y-x),y-x \rangle + o(t_{1}-s_{1}) + o\left(|y-x|^{2}+ \sum_{i=1}^{k}(t_{i}-s_{i})\right)
\end{split}
\]
as $(y,s) \to (x,t)$ such that $s_{i} < t_{i}$ for all $i=1,\ldots,k$. More precisely, for every $\eta>0$, one can find a sufficiently small compact neighbourhood $K_{i}$ of $(x_{i},t_{i})$ such that $s_{i} \leq t_{i}$ whenever $(y_{i},s_{i}) \in K_{i}$ and
\[
\begin{split}
u_{1}(y_{1},s_{1}) &+ \ldots + u_{k}(y_{k},s_{k}) \\&\leq u_{1}(x_{1},t_{1}) + \ldots + u_{k}(x_{k},t_{k}) \\&\hspace{0.5cm}+ \langle a,s-t \rangle  + \langle p,y-x \rangle + \frac{1}{2}\langle X(y-x),y-x \rangle \\&\hspace{0.5cm}+ \frac{\eta}{2}(t_{1}-s_{1}) + \ldots + \frac{\eta}{2}(t_{k}-s_{k}) + \frac{\eta}{2}|y-x|^{2}
\end{split}
\]
for all $(y,s) \in U_{1}\times\ldots\times U_{k}$.

If we define functions $\tilde{u}_{i}: \R^{d_{i}+1} \to \R \cup \{ -\infty \}$ by
\[
\tilde{u}_{i}(y_{i},s_{i}):= \begin{cases}
\hspace{-0.75cm}\begin{matrix}&u(y_{i},s_{i}) - u(x_{i},t_{i}) \\&\hspace{0.5cm}- a_{i} (s_{i} - t_{i}) - \langle p_{i} ,y_{i}-x_{i}\rangle\\&\hspace{0.5cm} - \frac{\eta}{2}(t_{i} - s_{i}) - \frac{\eta}{2}|y_{i}-x_{i}|^{2}\end{matrix} &\text{if $(y_{i},s_{i}) \in K_{i}$,}\\\\
- \infty & \text{if $(y_{i},s_{i}) \in\R^{d_{i}+1} \setminus K_{i}$,}
\end{cases}
\]
then
\[
\tilde{u}_{1}(y_{1},s_{1}) + \ldots + \tilde{u}_{k}(y_{k},s_{k}) \leq \frac{1}{2}\langle X(y-x),y-x \rangle 
\]
holds on $\R^{d+k}$. In particular, since $\tilde{u}_{i}(x_{i},t_{i})=0$ for each $i=1,\ldots,k$, we see that equality is attained at the point $(y,s)=(x,t)$. Therefore, the function \[\tilde{u}_{1}+\ldots+\tilde{u}_{k} - \frac{1}{2}\langle X(\cdot-x),\cdot-x \rangle\] attains a maximum at the point $(x,t)$. We note that the upper semicontinuity of $u$ and the compactness of $K_{i}$ will imply $\tilde{u}_{i} \in \USC(\R^{d_{i}+1})$ for all $i=1,\ldots,k$, and thus we may apply the maximum principle for semicontinuous functions by Crandall and Ishii \cite[Theorem 1]{Crandall-Ishii-maximum-semicontinuous} to conclude that for every $\kappa>0$, there exist matrices $\tilde{X}_{i,\eta} \in \sym_{d}$ such that
\[
(0,\tilde{X}_{i,\eta}) \in \overline{J}^{2,+}\tilde{u}_{i}(x_{i},t_{i})
\]
and
\[
-\left( \frac{1}{\kappa} + \|A\| \right)I_{d+k} \leq \begin{pmatrix}
\tilde{X}_{1,\eta} & \cdots & 0 \\ 
\vdots & \ddots & \vdots\\
0 & \cdots & \tilde{X}_{k,\eta}
\end{pmatrix} \leq A + \kappa A^{2},
\]
where
\[
A= \begin{pmatrix}
X_{11} &0 & \cdots & X_{1k}  & 0\\ 
0 & 0 & \cdots & 0 & 0\\
\vdots & \vdots & \ddots & \vdots & \vdots \\
X_{k1} & 0 & \cdots & X_{kk} & 0 \\
0 & 0 & \cdots & 0 & 0
\end{pmatrix}
\]
is the Hessian of $\frac{1}{2}\langle X(\cdot-x),\cdot-x \rangle$ in all variables $x_{1},t_{1},\ldots,x_{k},t_{k}$. Here $X_{ij}$ denotes the ${ij}^{th}$ block of the matrix $X \in \sym_{d}$ when partitioned in the natural way according to the decomposition $d=d_{1}+\ldots+d_{k}$. By considering the appropriate principle minor of the above matrices, this condition implies
\[
-\left( \frac{1}{\kappa} + \|A\| \right)I_{d} \leq \begin{pmatrix}
X_{1,\eta} & \cdots & 0 \\ 
\vdots & \ddots & \vdots\\
0 & \cdots & X_{k,\eta}
\end{pmatrix} \leq A + \kappa A^{2},
\]
where we now use $X_{i,\eta}$ to denote the truncation of the matrix $\tilde{X}_{i,\eta}$, which removes the last row and column. In addition, $(0,\tilde{X}_{i,\eta}) \in \overline{J}^{2,+}\tilde{u}_{i}(x_{i},t_{i})$ implies that $(0,0,X_{i,\eta}) \in \overline{\mathcal{P}}^{2,+}\tilde{u}_{i}(x_{i},t_{i})$. Indeed, by the definition of $(0,\tilde{X}_{i,\eta}) \in \overline{J}^{2,+}\tilde{u}_{i}(x_{i},t_{i})$, there are sequences \[(\tilde{p}_{i,\eta}^{(n)},\tilde{X}_{i,\eta}^{(n)})_{n \geq 1} \subseteq \R^{d_{i}+1} \times \sym_{d_{i}+1} \quad \text{and} \quad (x_{i,\eta}^{(n)})_{n \geq 1} \subseteq \R^{d_{i}+1}\]  such that 
\[
\begin{split}
    &(\tilde{p}_{i,\eta}^{(n)},\tilde{X}_{i,\eta}^{(n)}) \in J^{2,+}\tilde{u}_{i}(x_{i,\eta}^{(n)},t_{i,\eta}^{(n)}), \quad (\tilde{p}_{i,\eta}^{(n)},\tilde{X}_{i,\eta}^{(n)}) \to (0,\tilde{X}_{i,\eta})\\ &(x_{i,\eta}^{(n)},t_{i,\eta}^{(n)}) \to (x_{i},t_{i}), \quad \text{and} \quad  \tilde{u}_{i}(x_{i,\eta}^{(n)},t_{i,\eta}^{(n)}) \to \tilde{u}_{i}(x_{i},t_{i})
\end{split}
\]
as $n \to \infty$. Therefore
\begin{align*}
\tilde{u}_{i}(y_{i},s_{i}) &\leq \tilde{u}_{i}(x_{i,\eta}^{(n)},t_{i,\eta}^{(n)}) + \langle \tilde{p}_{i,\eta}^{(n)},(y_{i},s_{i})-(x_{i,\eta}^{(n)},t_{i,\eta}^{(n)})\rangle \\&\hspace{1cm}+ \frac{1}{2}\langle \tilde{X}_{i,\eta}^{(n)}((y_{i},s_{i})-(x_{i,\eta}^{(n)},t_{i,\eta}^{(n)})),(y_{i},s_{i})-(x_{i,\eta}^{(n)},t_{i,\eta}^{(n)}) \rangle \\&\hspace{1cm} +  o(|y_{i}-x_{i,\eta}^{(n)}|^{2}+(t_{i,\eta}^{(n)}-s_{i})^{2})\\
&= \tilde{u}_{i}(x_{i,\eta}^{(n)},t_{i,\eta}^{(n)}) + a_{i,\eta}^{(n)}(s_{i} - t_{i,\eta}^{(n)})+ \langle p_{i,\eta}^{(n)}, y_{i}-x_{i,\eta}^{(n)}\rangle \\&\hspace{1cm}+ \frac{1}{2}\langle X_{i,\eta}^{(n)}(y_{i}-x_{i,\eta}^{(n)}),y_{i}-x_{i,\eta}^{(n)} \rangle +  o(|y_{i}-x_{i,\eta}^{(n)}|^{2}+|t_{i,\eta}^{(n)}-s_{i}|)
\end{align*}
as $(y_{i},s_{i}) \to (x_{i,\eta}^{(n)},t_{i}^{(n)})$ in $\R^{d_{i}+1}$, where $\tilde{p}_{i,\eta}^{(n)} = (p_{i,\eta}^{(n)},a_{i,\eta}^{(n)})$ and $X_{i,\eta}^{(n)}$ is the truncation of $\tilde{X}_{i,\eta}^{(n)}$, which removes the final row and column. In particular, for $s_{i} \leq t_{i,\eta}^{(n)}$,
\[
\begin{split}
\tilde{u}_{i}(y_{i},s_{i}) &\leq \tilde{u}_{i}(x_{i,\eta}^{(n)},t_{i,\eta}^{(n)}) + a_{i,\eta}^{(n)}(s_{i}-t_{i,\eta}^{(n)}) + \langle p_{i,\eta}^{(n)},y_{i}-x_{i,\eta}^{(n)}\rangle \\&\hspace{1cm}+ \frac{1}{2}\langle X_{i,\eta}^{(n)}(y_{i}-x_{i,\eta}^{(n)}),y_{i}-x_{i,\eta}^{(n)} \rangle +  o(|y_{i}-x_{i,\eta}^{(n)}|^{2}+t_{i,\eta}^{(n)}-s_{i})
\end{split}
\]
as $y_{i} \to x_{i,\eta}^{(n)}$ in $\R^{d_{i}}$ and $s_{i} \to t_{i}^{(n)}$ in $\R$ from below. We note that since $\tilde{u}_{i}(x_{i},t_{i})$ is finite and $\tilde{u}_{i}$ was set to be $-\infty$ outside of $K_{i}$, there must be a subsequence of $(x_{i,\eta}^{(n)},t_{i,\eta}^{(n)})_{n \geq 1}$, which lies in $K_{i}$. In particular, we have $t_{i,\eta}^{(n)}  \leq t_{i}$. We work with this subsequence, again denoting it by $(x_{i,\eta}^{(n)},t_{i,\eta}^{(n)})_{n \geq 1}$. Then, it follows that $(a_{i,\eta}^{(n)},p_{i,\eta}^{(n)},X_{i,\eta}^{(n)}) \in \mathcal{P}^{2,+}\tilde{u}(x_{i,\eta}^{(n)},t_{i,\eta}^{(n)})$ for each $n \geq 1$, and so we have \linebreak $(0,0,X_{i,\eta}) \in \overline{\mathcal{P}}^{2,+}\tilde{u}_{i}(x_{i},t_{i})$.

Since the points $(x_{i,\eta}^{(n)},t_{i,\eta}^{(n)})$ and $(y_{i},s_{i})$ can be taken to lie in $K_{i}$, we use the definition of $\tilde{u}_{i}$ on $K_{i}$, to obtain
\[
\begin{split}
u_{i}(y_{i},s_{i})  &\leq 
u_{i}(x_{i,\eta}^{(n)},t_{i,\eta}^{(n)})+ \left( a_{i,\eta}^{(n)}+a_{i}-\frac{\eta}{2} \right)(s_{i}-t_{i,\eta}^{(n)})  + \langle p_{i}+p_{i,\eta}^{(n)},y_{i}-x_{i,\eta}^{(n)}\rangle \\&\hspace{1cm}+ \frac{1}{2}\langle X_{i,\eta}^{(n)}(y_{i}-x_{i,\eta}^{(n)}),y_{i}-x_{i,\eta}^{(n)} \rangle + \frac{\eta}{2}|y_{i}-x_{i}|^{2} - \frac{\eta}{2}|x_{i,\eta}^{(n)}-x_{i}|^{2} \\&\hspace{1cm}+ o(|y_{i}-x_{i,\eta}^{(n)}|^{2}+t_{i,\eta}^{(n)}-s_{i}).
\end{split}
\]
Rewriting $\frac{\eta}{2}|y_{i}-x_{i}|^{2} - \frac{\eta}{2}|x_{i,\eta}^{(n)}-x_{i}|^{2}$ as
\[
\begin{split}
 &\frac{\eta}{2}\left(|y_{i} - x_{i,\eta}^{(n)} + x_{i,\eta}^{(n)} -x_{i}|^{2}- |x_{i,\eta}^{(n)}-x_{i}|^{2}\right) \\&\hspace{0.5cm}= \frac{\eta}{2}|y_{i}-x_{i,\eta}^{(n)}|^{2}+ \eta \langle x_{i,\eta}^{(n)}-x_{i}, y_{i}-x_{i,\eta}^{(n)} \rangle  
\end{split}
\]
leads to
\[
\begin{split}
u_{i}(y_{i},s_{i})  &\leq 
u_{i}(x_{i,\eta}^{(n)},t_{i,\eta}^{(n)})+ \left( a_{i,\eta}^{(n)}+a_{i}-\frac{\eta}{2} \right)(s_{i}-t_{i,\eta}^{(n)})  \\&\hspace{1cm}+ \langle p_{i}+p_{i,\eta}^{(n)} + \eta  (x_{i,\eta}^{(n)}-x_{i}), y_{i}-x_{i,\eta}^{(n)}\rangle \\&\hspace{1cm}+ \frac{1}{2}\langle (X_{i,\eta}^{(n)} + \eta I_{d_{i}})(y_{i}-x_{i,\eta}^{(n)}),y_{i}-x_{i,\eta}^{(n)} \rangle \\&\hspace{1cm}+ o(|y_{i}-x_{i,\eta}^{(n)}|^{2}+t_{i,\eta}^{(n)}-s_{i}).
\end{split}
\]
This expresses that
\[
\left(a_{i,\eta}^{(n)}+a_{i}-\frac{\eta}{2},p_{i}+p_{i,\eta}^{(n)} + \eta(x_{i,\eta}^{(n)}-x_{i}),X_{i,\eta}^{(n)} + \eta I_{d_{i}}\right) \in \mathcal{P}^{2,+}u_{i}(x_{i,\eta}^{(n)},t_{i,\eta}^{(n)}).
\]
Considering the behaviour as $n \to \infty$, we find that
\[
\left(a_{i}-\frac{\eta}{2},p_{i},X_{i,\eta} + \eta I_{d_{i}}\right) \in \overline{\mathcal{P}}^{2,+}u_{i}(x_{i},t_{i}).
\]
Taking $\eta \to 0$ as in the proof of Crandall and Ishii, we find that there is a matrix $X_{i} \in \sym_{d_{i}}$ such that $X_{i,\eta} \to X_{i}$ and so we have that $(a_{i},p_{i},X_{i}) \in \overline{\mathcal{P}}^{2,+}u_{i}(x_{i},t_{i})$.
\end{proof}

Using our parabolic theorem on sums, we deduce the following statement. 

\begin{corollary}[Parabolic maximum principle for semicontinuous functions I]\label{cor:para-max-semi}
For $i=1,\ldots,k$, let $\Omega_{i} \subseteq \R^{d_{i}}$ be locally compact, and set $\Omega:= \Omega_{1}\times\ldots\times \Omega _{k}$ and $S:= (S_{1},T_{1}] \times \ldots \times (S_{k},T_{k}]$. Let $u_{i} \in \USC(\overline{\Omega_{i}}\times[S_{i},T_{i}])$ and let $\varphi:\Omega \times S \to \R$, $\varphi:=\varphi(x_{1},\ldots,x_{k},t_{1},\ldots,t_{k})$ be twice continuously differentiable in each of the variables $x_{i}$ and continuously differentiable in the variables $t_{i}$ for $i=1,\ldots,k$. Consider the function
\[
Z(x,t) := u_{1}(x_{1},t_{1}) + \ldots + u_{k}(x_{k},t_{k}) - \varphi(x,t)
\]
defined for all $x:=(x_{1},\ldots,x_{k})^{\top} \in \Omega$ and $t:=(t_{1},\ldots,t_{k})^{\top} \in S$. If there exists a point $(x_{0},t_{0}) \in \Omega \times S$ such that $Z$ touches the value $Z(x_0,t_0)$ from below for the first time at $(x_0,t_0)$, then for all $\kappa >0$, there are matrices $X_{i} \in \sym_{d_{i}}$ such that for all $i=1,\ldots,k$,
\[
(\partial_{t_{i}}\varphi(x_{0},t_{0}),\nabla_{x_{i}}\varphi(x_{0},t_{0}),X_{i}) \in \overline{\mathcal{P}}^{2,+}u_{i}(x_{0i},t_{0i})
\]
and
\[
-\left( \frac{1}{\kappa} + \|X\| \right)I_{d} \leq \begin{pmatrix}
X_{1} & \cdots & 0 \\ 
\vdots & \ddots & \vdots\\
0 & \cdots & X_{k}
\end{pmatrix} \leq X + \kappa X^{2},
\]
where $X=D^{2}_{x}\varphi(x_{0},t_{0})$ and $d:=d_{1}+\ldots+d_{k}$.
\end{corollary}

\begin{proof}
The premise of Corollary \ref{cor:para-max-semi} is described by the inequality
\[
u(x,t) \leq u(x_{0},t_{0}) + \varphi(x,t) - \varphi(x_{0},t_{0})
\]
holding for $x \in \Omega$ and $t$ with $t_{i} \leq t_{0,i}$. Since $\varphi \in C^{2}(Q)$, we may apply the Taylor expansion of $\varphi$ to obtain
\[
\begin{split}
  u(x,t) \leq u(x_{0},&t_{0}) + \langle \nabla_{t}\varphi(x_{0},t_{0}),t-t_{0}\rangle + \langle \nabla_{x}\varphi(x_{0},t_{0}),x-x_{0}\rangle \\&+ \frac{1}{2} \langle D^{2}_{x}\varphi(x_{0},t_{0})(x-x_{0}),x-x_{0} \rangle + o\left(|t_{0}-t|_1 +  |x-x_{0}|^{2} \right)  
\end{split}
\]
as $x \to x_{0}$ in $\Omega$ and $t_{i} \to t_{0i}^{-}$, which exactly expresses that
\[
(\nabla_{t}\varphi(x_{0},t_{0}),\nabla_{x}\varphi(x_{0},t_{0}),D^{2}_{x}\varphi(x_{0},t_{0})) \in \mathcal{P}^{2,+}u(x_{0},t_{0}).
\]
The conclusion now follows immediately from Theorem \ref{thm:para-thm-sums}.
\end{proof}

We also state the following variation of Corollary \ref{cor:para-max-semi}, which will be more readily applicable to the proof of the Harnack inequality, in which we typically consider a function attaining a minimum rather than a maximum. The corollary that follows can be obtained using the relation $\mathcal{P}^{2,-}u(x_{0},t_{0}) = -\mathcal{P}^{2,+}(-u)(x_{0},t_{0})$.

\begin{corollary}[Parabolic maximum principle for semicontinuous functions II]\label{cor:para-max-semi-min}
For $i=1,\ldots,k$, let $\Omega_{i} \subseteq \R^{d_{i}}$ be locally compact, and set $\Omega:= \Omega_{1}\times\ldots\times \Omega _{k}$ and $S:= (S_{1},T_{1}] \times \ldots \times (S_{k},T_{k}]$. Let $u_{i} \in \LSC(\overline{\Omega_{i}}\times[S_{i},T_{i}])$ and let $\varphi:\Omega \times S \to \R$, $\varphi:=\varphi(x_{1},\ldots,x_{k},t_{1},\ldots,t_{k})$ be twice continuously differentiable in each of the variables $x_{i}$ and continuously differentiable in the variables $t_{i}$ for $i=1,\ldots,k$. Consider the function
\[
Z(x,t) := u_{1}(x_{1},t_{1}) + \ldots + u_{k}(x_{k},t_{k}) - \varphi(x,t)
\]
defined for all $x:=(x_{1},\ldots,x_{k})^{\top} \in \Omega$ and $t:=(t_{1},\ldots,t_{k})^{\top} \in S$. If there exists a point $(x_{0},t_{0}) \in \Omega \times S$ such that $Z$ touches the value $Z(x_0,t_0)$ from above for the first time at $(x_0,t_0)$, then for all $\kappa >0$, there are matrices $X_{i} \in \sym_{d_{i}}$ such that for all $i=1,\ldots,k$,
\[
(\partial_{t_{i}}\varphi(x_{0},t_{0}),\nabla_{x_{i}}\varphi(x_{0},t_{0}),X_{i}) \in \overline{\mathcal{P}}^{2,-}u_{i}(x_{0i},t_{0i})
\]
and
\[
\left( \frac{1}{\kappa} + \|X\| \right)I_{d} \geq \begin{pmatrix}
X_{1} & \cdots & 0 \\ 
\vdots & \ddots & \vdots\\
0 & \cdots & X_{k}
\end{pmatrix} \geq X - \kappa X^{2},
\]
where $X=D^{2}_{x}\varphi(x_{0},t_{0})$.
\end{corollary}

\section{Harnack inequalities for viscosity solutions}\label{sec:harnack-viscosity}

In this section, we demonstrate how to employ our multi-point maximum principle approach to prove Harnack inequalities satisfied by viscosity solutions of parabolic equations. We begin by discussing viscosity solutions of the linear heat equation to first illuminate the core ideas of the proof before we introduce the computational complexities encountered when dealing with the doubly nonlinear heat equation, which we will discuss in the second part of this section. We assume throughout that $\Omega \subseteq \R^d$ is a bounded, convex domain.

\subsection{The linear heat equation}\label{subsec:schroe-visc}

Here, we are concerned with viscosity solutions of the Neumann problem 
\begin{equation}\label{2025-10-21: 
1}
 \begin{cases}
   \begin{alignedat}{2}
       \partial_t u &= \Delta u&& \qquad \text{in $\Omega \times (0,T)$,} \\
        \nabla u \cdot \nu &= 0 && \qquad \text{on $\partial \Omega \times (0,T)$}
   \end{alignedat}
\end{cases}
\end{equation}
for the heat equation.

\begin{proof}[Proof of Theorem \ref{thm:Euclidean-Schroedinger-full-visc}]
Let $u$ be a positive viscosity solution of \eqref{2025-10-21: 1} and set $v=\log u$. Then $v$ is a viscosity solution of 
\begin{equation}\label{2025-10-21: 
2}
 \begin{cases}
   \begin{alignedat}{2}
       \partial_t v &= \Delta v + |\nabla v|^2&& \qquad \text{in $\Omega \times (0,T)$,} \\
        \nabla v \cdot \nu &= 0 && \qquad \text{on $\partial \Omega \times (0,T)$.}
   \end{alignedat}
\end{cases}
\end{equation}

Fix $\varepsilon > 0$ and define a function $Z$ on $\overline{\Omega} \times \overline{\Omega} \times S$ by
\[
Z(x,y;t,s)=v(x,t) - v(y,s) + \log\left(\frac{\beta(t)}{\beta(s)}\right) + \omega(x,y;t,s) + \frac{\varepsilon}{2}(t-s)^2 + \varepsilon.
\]
Identically as in the case for classical solutions, we know that $Z>0$ whenever $s$ is sufficiently small, or $s$ is sufficiently close to $t$. In particular, $Z>0$ in a small enough neighbourhood of $(0,0)$. Suppose there is a point $P_{0}:=(x_{0},y_{0};t_{0},s_{0})$ such that $Z=0$ for the first time at $P_0$ in the sense of Definition \ref{def:first-touching-point}. We first assume that $x_{0},y_{0} \in \Omega$.

Applying Corollary \ref{cor:para-max-semi-min} with $u_{1}(x,t)=v(x,t)$, $u_{2}(y,s)=-v(y,s)$, and \[\varphi(x,y;t,s)= -\log\left(\frac{\beta(t)}{\beta(s)}\right) - \omega(x,y;t,s) - \frac{\varepsilon}{2}(t-s)^2 - \varepsilon\] for $(x,y,t,s) \in \Omega \times \Omega \times (s_{0},t_{0}]\times (0,s_{0}]$, we deduce that for every $\kappa>0$, there exist matrices $X,Y \in \sym_{d}$ such that
\begin{subequations}
  \begin{align}
(\partial_{t}\varphi(P_{0}),\nabla_{x}\varphi(P_{0}),X) &\in \overline{\mathcal{P}}^{2,-}v(x_{0},t_{0})\label{2025-10-21: 
3} \\
(\partial_{s}\varphi(P_{0}),\nabla_{y}\varphi(P_{0}),-Y) &\in \overline{\mathcal{P}}^{2,-}(-v)(y_{0},s_{0})\label{2025-10-21: 
4}\\
 \left( \frac{1}{\kappa} + \|H\| \right)I_{2d} \geq \begin{pmatrix}
X  & 0 \\ 
0 & -Y
\end{pmatrix} &\geq H - \kappa H^{2},\label{2025-10-21: 
5}
\end{align}  
\end{subequations}
where $H= D^{2}_{(x,y)}\varphi(P_{0}) = -D^{2}_{(x,y)}\omega(P_{0})$. Computing the first-order derivatives of $\varphi$ explicitly, the condition \eqref{2025-10-21: 3} implies
\[
\left( -\frac{\beta'(t_{0})}{\beta(t_{0})} - \partial_{t} \omega(P_{0})- \varepsilon(t_{0}-s_{0}), -\nabla_{x}\omega(P_{0}),X \right) \in \overline{\mathcal{P}}^{2,-}v(x_{0},t_{0}).
\]
Using that $v$ is a supersolution of \eqref{2025-10-21: 2}, this implies
\[
-\frac{\beta'(t_{0})}{\beta(t_{0})} - \partial_{t} \omega(P_{0})- \varepsilon(t_{0}-s_{0}) - \trace X - |\nabla_{x}\omega(P_{0})|^{2} \geq 0.
\]
However, by \eqref{eq:schroe-omega-ineq-1-visc}, this reduces to
\begin{equation}\label{2025-10-21: 
6}
 \frac{\beta'(t_{0})}{\beta(t_{0})} + \varepsilon(t_{0}-s_{0}) + \trace X \leq 0.
\end{equation}
In a similar way, we understand \eqref{2025-10-21: 4} to mean that
\[
\left( \frac{\beta'(s_{0})}{\beta(s_{0})} - \partial_{s} \omega(P_{0})+\varepsilon(t_{0}-s_{0}), -\nabla_{y}\omega(P_{0}),-Y \right) \in \overline{\mathcal{P}}^{2,-}(-v)(y_{0},s_{0}),
\]
which is equivalent to
\[
\left( -\frac{\beta'(s_{0})}{\beta(s_{0})} + \partial_{s} \omega(P_{0})-\varepsilon(t_{0}-s_{0}), \nabla_{y}\omega(P_{0}),Y \right) \in \overline{\mathcal{P}}^{2,+}v(y_{0},s_{0}).
\]
Since $v$ is also a subsolution of \eqref{2025-10-21: 2}, we have
\[
-\frac{\beta'(s_{0})}{\beta(s_{0})} + \partial_{s} \omega(P_{0})-\varepsilon(t_{0}-s_{0})-\trace Y - |\nabla_{y}\omega(P_{0})|^{2} \leq 0,
\]
which, because of \eqref{eq:schroe-omega-ineq-2-visc} becomes
\begin{equation}\label{2025-10-21: 
7}
- \frac{\beta'(s_{0})}{\beta(s_{0})} -\varepsilon(t_{0}-s_{0})-\trace Y  \leq 0.
\end{equation}

Taking a linear combination of the inequalities \eqref{2025-10-21: 6} and \eqref{2025-10-21: 7} with the coefficients $A(t_{0})^{2}$ and $A(s_{0})^{2}$ yields
\[
\begin{split}
    0 &\geq A(t_{0})^{2} \frac{\beta'(t_{0})}{\beta(t_{0})} - A(s_{0})^{2}\frac{\beta'(s_{0})}{\beta(s_{0})}+ \varepsilon(t_{0}-s_{0})(A(t_{0})^{2}-A(s_{0})^{2}) \\& \hspace{1cm}+ A(t_{0})^{2}\trace X- A(s_{0})^{2}\trace Y.
\end{split}
\]
Analysing first the expression $A(t_{0})^{2}\trace X- A(s_{0})^{2}\trace Y$ using inequality \eqref{2025-10-21: 5}, we note that
\begin{align*}
A(t_{0})^{2}\trace X- A(s_{0})^{2}\trace Y&=\trace(A(t_{0})^{2}X - A(s_{0})^{2}Y)  \\
&= \trace \begin{pmatrix}
A(t_{0})^{2}I_{d} & A(t_{0})A(s_{0})I_{d} \\
A(t_{0})A(s_{0})I_{d} & A(s_{0})^{2}I_{d}
\end{pmatrix} \begin{pmatrix}
X & 0 \\
0 & -Y
\end{pmatrix} \\
& \geq \trace \begin{pmatrix}
A(t_{0})^{2}I_{d} & A(t_{0})A(s_{0})I_{d} \\
A(t_{0})A(s_{0})I_{d} & A(s_{0})^{2}I_{d}
\end{pmatrix} (H - \kappa H^{2}).
\end{align*}
Therefore
\[
\begin{split}
    0 &\geq A(t_{0})^{2} \frac{\beta'(t_{0})}{\beta(t_{0})} - A(s_{0})^{2}\frac{\beta'(s_{0})}{\beta(s_{0})}+ \varepsilon(t_{0}-s_{0})(A(t_{0})^{2}-A(s_{0})^{2}) \\& \hspace{1cm}+ \trace \begin{pmatrix}
A(t_{0})^{2}I_{d} & A(t_{0})A(s_{0})I_{d} \\
A(t_{0})A(s_{0})I_{d} & A(s_{0})^{2}I_{d}
\end{pmatrix} (H - \kappa H^{2})
\end{split}
\]
for all $\kappa>0$. Taking $\kappa \to 0$ and applying \eqref{eq:schroe-omega-ineq-3-visc} leads to
\begin{align*}
0 &\geq A(t_{0})^{2} \frac{\beta'(t_{0})}{\beta(t_{0})} - A(s_{0})^{2}\frac{\beta'(s_{0})}{\beta(s_{0})}+ \varepsilon(t_{0}-s_{0})(A(t_{0})^{2}-A(s_{0})^{2}) \\&\hspace{1cm}+ \trace \begin{pmatrix}
A(t_{0})^{2}I_{d} & A(t_{0})A(s_{0})I_{d} \\
A(t_{0})A(s_{0})I_{d} & A(s_{0})^{2}I_{d}
\end{pmatrix} \begin{pmatrix} -D^{2}_{x}\omega(P_{0}) & -D^{2}_{xy}\omega (P_{0}) \\
-D^{2}_{xy}\omega(P_{0}) & -D^{2}_{y}\omega(P_{0})
\end{pmatrix} \\
&= A(t_{0})^{2} \frac{\beta'(t_{0})}{\beta(t_{0})} - A(s_{0})^{2}\frac{\beta'(s_{0})}{\beta(s_{0})}+ \varepsilon(t_{0}-s_{0})(A(t_{0})^{2}-A(s_{0})^{2}) \\&\hspace{1cm}- A(t_{0})^{2}\Delta_{x}\omega(P_{0})- A(s_{0})^{2}\Delta_{y}\omega(P_{0})-2A(t_{0})A(s_{0})\Delta_{xy}\omega(P_{0}) \\
& \geq \varepsilon(t_{0}-s_{0})(A(t_{0})^{2}-A(s_{0})^{2}) >0,
\end{align*}
which is a contradiction. Therefore, if such a point $P_{0}$ exists, it must occur with at least one of $x_{0}$ or $y_{0}$ on the boundary $\partial \Omega$. We proceed in this case with an approach similar to the one applied in \cite[Section 6.2]{li-wang-mod}.

By definition of the point $P_{0}=(x_{0},y_{0};t_{0},s_{0})$, there exists $0 < \eta < t_{0}-s_{0}$ such that
\[
Z(x_{0},y_{0};t_{0},s_{0}) \leq Z(x,y;t,s)
\]
for all $(x,y;t,s) \in K:= \overline{\Omega}\times \overline{\Omega} \times [t_{0}- \eta] \times [s_{0}- \eta , s_{0}]$, that is, $Z$ attains its minimum on $K$ at $P_{0}$. We can ensure that $P_{0}$ is the unique minimiser of $Z$ on $K$ by modifying this function. We define
\[
\tilde{Z}(x,y;t,s):=Z(x,y;t,s) + |x-x_{0}|^{4} + |y-y_{0}|^{4} + |t-t_{0}|^{4} + |s-s_{0}|^{4}
\]
for all $(x,y;t,s) \in K$.

Before we can begin the main step of the argument, we first perturb the solution $v$, and in doing so, we perturb the equation satisfied by $v$ as well. Fix a point $z_{0} \in \Omega$ so that $\delta:= d(z_{0},\partial \Omega)>0$ and define
\[
h(x):= \frac{1}{2}|x-z_{0}|^{2}\qquad \text{for all $x \in \overline{\Omega}$}.
\]
Then for $\theta >0$ small, we define
\[
v^{\theta}(x,t):= v(x,t) + \theta h(x) \quad \text{and} \quad v_{\theta}(x,t):=v(x,t) - \theta h(x)
\]
for all $(x,t) \in \overline{\Omega} \times (0,T)$. A simple calculation shows that $v^{\theta}$ satisfies
\begin{equation}\label{2025-11-14: 
v-upper-prob}
 \begin{cases}
   \begin{alignedat}{2}
       \partial_t v^{\theta} - \Delta (v^{\theta} - \theta h) - |\nabla(v^{\theta} - \theta h)|^{2} &= 0&& \qquad \text{in $\Omega \times (0,T)$,} \\
        \nabla v^{\theta} \cdot \nu - \theta \nabla h \cdot \nu &= 0 && \qquad \text{on $\partial \Omega \times (0,T)$,}
   \end{alignedat}
\end{cases}
\end{equation}
and $v_{\theta}$ satisfies
\begin{equation}\label{2025-11-14: 
v-lower-prob}
 \begin{cases}
   \begin{alignedat}{2}
       \partial_t v_{\theta} - \Delta (v_{\theta} + \theta h) - |\nabla(v_{\theta} + \theta h)|^{2} &= 0&& \qquad \text{in $\Omega \times (0,T)$,} \\
        \nabla v_{\theta} \cdot \nu + \theta \nabla h \cdot \nu &= 0 && \qquad \text{on $\partial \Omega \times (0,T)$,}
   \end{alignedat}
\end{cases}
\end{equation}
both in the sense of viscosity solutions. Using the functions $v^{\theta}$ and $v_{\theta}$, we then define an approximation of $\tilde{Z}$ by setting
\[
\begin{split}
   \tilde{Z}_{\theta}(x,y;t,s)&:= v^{\theta}(x,t) - v_{\theta}(y,s) + \log \left( \frac{\beta(t)}{\beta (s)} \right) + \omega(x,y;t,s) + \frac{\varepsilon}{2}(t-s)^{2}+ \varepsilon \\&\hspace{1cm}+ |x-x_{0}|^{4} + |y-y_{0}|^{4} + |t-t_{0}|^{4} + |s-s_{0}|^{4} 
\end{split}
\]
for all $(x,y;t,s) \in K$. Since for each $\theta>0$, $\tilde{Z}_{\theta}$ is a continuous function on the compact set $K$, there exists $P_{\theta}:=(x_{\theta},y_{\theta};t_{\theta},s_{\theta}) \in K$ such that
\[
\tilde{Z}_{\theta}(P_{\theta})= \min_{P \in K}Z(P).
\]
We claim that since $\tilde{Z}_{\theta} \to \tilde{Z}$ uniformly on $K$ as $\theta \to 0$ and because $P_{0}$ is the unique minimiser of $\tilde{Z}$ in $K$, there is a sequence $(\theta _{n})_{n \geq 1}$ of positive numbers such that $P_{\theta_{n}} \to P_{0}$ as $n \to \infty$. Since $(P_{\theta})_{\theta >0}$ is contained in the bounded set $K$, there is a sequence $(\theta _{n})_{n \geq 1}$ and a point $\hat{P} \in K$ such that $P_{\theta _{n}} \to \hat{P}$ in $K$. Then $\tilde{Z}_{\theta_{n}}(P_{\theta_{n}}) \to \tilde{Z}(\hat{P})$ as $n \to \infty$ since
\begin{align*}
|\tilde{Z}_{\theta_{n}}(P_{\theta_{n}}) - \tilde{Z}( \hat{P})| &\leq | \tilde{Z}_{\theta _{n}} - \tilde{Z}(P_{\theta _{n}})| + | \tilde{Z}(P_{\theta _{n}})- \tilde{Z}( \hat{P})| \\
& \leq \|\tilde{Z}_{\theta _{n}} - \tilde{Z}\|_{\infty} + | \tilde{Z}(P_{\theta _{n}}) - \tilde{Z}(\hat{P})| \to 0
\end{align*}
because $\tilde{Z}_{\theta _{n}} \to \tilde{Z}$ uniformly as $n \to \infty$ and $\tilde{Z}$ is continuous. Now, by the definition of $P_{\theta _{n}}$, one has that
\[
\tilde{Z}_{\theta _{n}}(P_{\theta _{n}}) \leq \tilde{Z}_{\theta _{n}}(P)
\]
for all points $P \in K$. Taking $n \to \infty$ gives
\[
\tilde{Z}(\hat{P}) \leq \tilde{Z}(P)
\]
for all $P \in K$, which expresses that $\hat{P}$ minimises $Z$ over $K$. However, $\tilde{Z}$ was defined in such a way that $P_{0}$ is the unique minimiser of $\tilde{Z}$ over $K$, and $\hat{P}= P_{0}$ and $P_{\theta _{n}} \to P_{0}$ as claimed. Moreover, since $(t_{\theta _{n}},s_{\theta _{n}}) \to (t_{0},s_{0})$, we may assume the point $P_{\theta _{n}}$ does not occur with $t_{\theta _{n}} = t_{0}-\eta$ or $s_{\theta _{n}} = s_{0}-\eta$. Therefore, $\tilde{Z}_{\theta _{n}}$ attains a minimum at $P_{\theta _{n}} \in \tilde{K}:= \overline{\Omega} \times \overline{\Omega} \times (t_{0}- \eta] \times (s_{0}- \eta , s_{0}]$. Thus, we may apply the parabolic maximum principle for semicontinuous functions (Corollary \ref{cor:para-max-semi-min}) to the function
\[
\tilde{Z}_{\theta _{n}}(x,y;t,s)= v^{\theta_{n}}(x,t) - v_{\theta _{n}}(y,s) - \varphi(x,y;t,s)
\]
at $P_{\theta _{n}}$, where
\[
\begin{split}
  \varphi(x,y;t,s)&= - \log \left( \frac{\beta(t)}{\beta (s)} \right) - \omega(x,y;t,s) - \frac{\varepsilon}{2}(t-s)^{2}- \varepsilon \\& \hspace{1cm}- |x-x_{0}|^{4} - |y-y_{0}|^{4} - |t-t_{0}|^{4} - |s-s_{0}|^{4}.  
\end{split}
\]
Then for every $\kappa >0$, there are matrices $X_{n}, Y_{n}$ such that
\begin{align}
(\partial_{t}\varphi(P_{\theta _{n}}),\nabla_{x}\varphi(P_{\theta _{n}}),X_{n}) &\in \overline{\mathcal{P}}^{2,-}_{\tilde{K}}v^{\theta_{n}}(x_{\theta _{n}},t_{\theta _{n}}), \\
(-\partial_{s}\varphi(P_{\theta _{n}}),-\nabla_{y}\varphi(P_{\theta _{n}}),Y_{n}) &\in \overline{\mathcal{P}}^{2,+}_{\tilde{K}}v_{\theta _{n}}(y_{\theta _{n}},s_{\theta _{n}}),\\
 \left( \frac{1}{\kappa} + \|H_{n}\| \right)I_{2d} \geq \begin{pmatrix}
X_{n}  & 0 \label{eq:matrix-ineq-n}\\ 
0 & -Y_{n}
\end{pmatrix} &\geq H_{n} - \kappa H_{n}^{2},
\end{align}
where $H_{n} = D^{2}_{(x,y)}\varphi(P_{\theta_{n}})$. We now use that $v^{\theta _{n}}$ is a viscosity supersolution of \eqref{2025-11-14: v-upper-prob}. If $x_{\theta _{n}} \in \Omega$, we have
\begin{equation}\label{2025-11-14: 
positive-thing-I}
\begin{split}
    Q_{n}&:= - \frac{\beta'(t_{\theta _{n}})}{\beta(t_{\theta _{n}})} - \partial_{t} \omega(P_{\theta _{n}}) - \varepsilon(t_{\theta _{n}} - s_{\theta _{n}}) - 4|t_{\theta _{n}} - t_{0}|^{2}(t_{\theta _{n}} - t_{0}) - \trace X_{n} \\&\hspace{1cm}- |-\nabla_{x}\omega(P_{\theta _{n}}) - 4|x_{\theta _{n}} - x_{0}|^{2}(x_{\theta _{n}} - x_{0}) -\theta _{n}\nabla h(x_{\theta _{n}})|^{2}\\&\hspace{1cm} + \theta _{n} \Delta h(x_{\theta _{n}}) \geq 0,
\end{split}
\end{equation}
whereas if $x_{\theta _{n}} \in \partial \Omega$, we have
\begin{equation}\label{2025-11-14: 
max}
 \max \left\{\begin{aligned}
     Q_{n}, &- \nabla_{x} \omega(P_{\theta _{n}})\cdot \nu(x_{\theta _{n}}) - 4|x_{\theta _{n}} - x_{0}|^{2}(x_{\theta _{n}} - x_{0})\cdot \nu(x_{\theta _{n}}) \\&\hspace{1cm}- \theta _{n} \nabla h(x_{\theta_{n}})\cdot \nu(x_{\theta_{n}})
 \end{aligned}   \right\} \geq 0.
\end{equation}
However, $\nabla_{x} \omega(P_{\theta _{n}})\cdot \nu(x_{\theta _{n}}) \geq 0$ by assumption, $(x_{\theta _{n}} - x_{0})\cdot \nu(x_{\theta _{n}}) \geq 0$ since $\Omega$ is convex, and finally
\[
\nabla h(x_{\theta_{n}})\cdot \nu(x_{\theta_{n}}) = (x_{\theta _{n}} - z_{0})\cdot \nu(x_{\theta _{n}}) \geq d(z_{0},\partial \Omega)= \delta.
\]
Therefore
\[
- \nabla_{x} \omega(P_{\theta _{n}})\cdot \nu(x_{\theta _{n}}) - 4|x_{\theta _{n}} - x_{0}|^{2}(x_{\theta _{n}} - x_{0})\cdot \nu(x_{\theta _{n}}) - \theta _{n} \nabla h(x_{\theta_{n}})\cdot \nu(x_{\theta_{n}}) \leq - \theta _{n}\delta <0.
\]
Hence, in order for \eqref{2025-11-14: max} to be satisfied, we must have that $Q_{n} \geq 0$, and so \eqref{2025-11-14: positive-thing-I} holds regardless of whether $x_{\theta _{n}} \in \Omega$ or $x_{\theta _{n}} \in \partial \Omega$. Applying a similar argument using that $v_{\theta _{n}}$ is a viscosity subsolution of \eqref{2025-11-14: v-lower-prob}, we find that
\begin{equation}\label{2025-11-14: 
positive-thing-II}
\begin{split}
    &\frac{\beta'(s_{\theta _{n}})}{\beta(s_{\theta _{n}})} - \partial_{s} \omega(P_{\theta _{n}}) + \varepsilon(t_{\theta _{n}} - s_{\theta _{n}}) - 4|s_{\theta _{n}} - s_{0}|^{2}(s_{\theta _{n}} - s_{0}) + \trace Y_{n} \\&\hspace{1cm}+ |-\nabla_{y}\omega(P_{\theta _{n}}) - 4|y_{\theta _{n}} - y_{0}|^{2}(y_{\theta _{n}} - y_{0}) +\theta _{n}\nabla h(y_{\theta _{n}})|^{2}\\&\hspace{1cm} - \theta _{n} \Delta h(y_{\theta _{n}}) \geq 0. 
\end{split}
\end{equation}
Seeking to take $n \to \infty$ in \eqref{eq:matrix-ineq-n}, \eqref{2025-11-14: positive-thing-I}, and \eqref{2025-11-14: positive-thing-II}, we note that
\begin{align*}
H_{n}&= D^{2}_{(x,y)}\varphi(P_{\theta_{n}}) \\&= - D^{2}_{(x,y)}\omega(P_{\theta_{n}}) -4\begin{pmatrix}
2(x_{\theta _{n}} - x_{0})(x_{\theta _{n}} - x_{0})^{\top} & 0 \\
0&2(y_{\theta _{n}} - y_{0})(y_{\theta _{n}} - y_{0})^{\top}
\end{pmatrix} \\& \hspace{1cm}-4 \begin{pmatrix}
|x_{\theta _{n}} - x_{0}|^{2}I_{d} & 0 \\
0&|y_{\theta _{n}} - y_{0}|^{2}I_{d}
\end{pmatrix}
\end{align*}
converges to $D^{2}_{(x,y)}\omega(P_{0})$ as $n \to \infty$, the sequences $(X_{n})_{n \geq 1}$ and $(Y_{n})_{n \geq 1}$ converge, possibly after taking a subsequence, because of \eqref{eq:matrix-ineq-n}, to some matrices \linebreak $X,Y \in \sym_{d}$ respectively. It follows that
\begin{align*}
-\frac{\beta'(t_{0})}{\beta(t_{0})} - \partial_{t} \omega(P_{0})- \varepsilon(t_{0}-s_{0}) - \trace X - |\nabla_{x}\omega(P_{0})|^{2} &\geq 0 \\
\frac{\beta'(s_{0})}{\beta(s_{0})} - \partial_{s} \omega(P_{0})+\varepsilon(t_{0}-s_{0})+\trace Y + |\nabla_{y}\omega(P_{0})|^{2} &\geq 0 \\
\left( \frac{1}{\kappa} + \|H\| \right)I_{2d} \geq \begin{pmatrix}X  & 0 \\ 0 & -Y \end{pmatrix} \geq H - \kappa H^{2}&
\end{align*}
hold with $H = -D^{2}_{(x,y)}\omega(P_{0})$ as before. Therefore, we may continue using the same argument as when $x_{0},y_{0} \not\in \partial \Omega$ to reach a contradiction once again. Thus, no such point $P_{0}$ exists and the proof is complete.

\end{proof}

\subsection{The doubly nonlinear heat equation}\label{subsec:dne-visc}
We now return to the Neumann problem for the doubly nonlinear heat equation with $m>0$ and $p \geq 2$, with emphasis on the case $\lambda > 0$. Our proof of Theorem \ref{thm:dne-viscosity} proceeds in a similar fashion to the proof of Theorem \ref{thm:Euclidean-Schroedinger-full-visc}, except we alter the definition of the function $Z$, so that for any positive (viscosity) solution $u$ of \eqref{2025-10-27: 1} we define $Z$ by
\[
Z(x,y;t,s):= \beta^{\lambda}(t) v(x,t) - \beta^{\lambda}(s) v(y,s) + \omega(x,y;t,s) + \frac{\varepsilon}{2}(t-s)^2 + \varepsilon,
\]
where $v= \frac{m}{\lambda}u^{\lambda}$. However, our application of the parabolic maximum principle for semicontinuous functions (Corollary \ref{cor:para-max-semi-min}) then only initially provides elements of $\overline{\mathcal{P}}^{2,\pm}(\beta^{\lambda}v)$ instead of $\overline{\mathcal{P}}^{2,\pm}v$. Therefore, we must first understand the following ``product rule'' type results.

\begin{lemma}[Product rule for parabolic semijets]\label{lem:product-rule}
Let $u \in \USC(\Omega \times(S_i,T_i))$ or $u \in \LSC(\Omega \times(S_i,T_i))$, and $\mu:[0,\infty) \to [0,\infty)$. 
\begin{enumerate}
    \item If $\mu$ is differentiable at $t_{0} \in (S_i,T_i)$ and $(a,\xi,X) \in \mathcal{P}^{2,\pm}u(x_{0},t_{0})$ for   $(x_{0},t_{0}) \in \Omega \times(S_i,T_i)$, then 
\[
(a \mu(t_{0}) + \mu'(t_{0})u(x_{0},t_{0}),\mu(t_{0})\xi,\mu(t_{0})X) \in \mathcal{P}^{2,\pm}(\mu(t_{0})u(x_{0},t_{0})).
\]
\item If $\mu$ is continuously differentiable and  $(a,\xi,X) \in \overline{\mathcal{P}}^{2,\pm}u(x_{0},t_{0})$ for some $(x_{0},t_{0}) \in \Omega \times(0,T)$, then 
\[
(a \mu(t_{0}) + \mu'(t_{0})u(x_{0},t_{0}),\mu(t_{0})\xi,\mu(t_{0})X) \in \overline{\mathcal{P}}^{2,\pm}(\mu(t_{0})u(x_{0},t_{0})).
\]
\item If $\mu$ is continuously differentiable,   $(a,\xi,X) \in \overline{\mathcal{P}}^{2,\pm}(\mu(x_{0},t_{0})u(x_{0},t_{0}))$ and $\mu(t_{0}) \ne 0$, then 
\[
\left( \frac{a}{\mu(t_{0})}-u(x_{0},t_{0}) \frac{\mu'(t_{0})}{\mu(t_{0})}, \frac{1}{\mu(t_{0})}\xi, \frac{1}{\mu(t_{0})}X \right) \in \overline{\mathcal{P}}^{2,\pm }u(x_{0},t_{0}).
\]
\end{enumerate}
\end{lemma}

\begin{proof}
\begin{enumerate}    
\item[]
\item We give the details of the proof for $\mathcal{P}^{2+}u$ and the claim about $\mathcal{P}^{2,-}u$ follows similarly. Since $\mu$ is differentiable at $t_{0}$,
\[
\mu(t) = \mu(t_{0}) + \mu'(t_{0})(t-t_{0}) + o(|t-t_{0}|)
\]
as $t \to t_{0}$, and since $(a,\xi,X) \in \mathcal{P}^{2,+}u(x_{0},t_{0})$,
\[
\begin{split}
u(x,t) &\leq u(x_{0},t_{0}) + a(t-t_{0}) + \langle \xi ,x-x_{0}\rangle\\&\hspace{1cm}+ \frac{1}{2} \langle X(x-x_{0}),x-x_{0} \rangle + o(t_{0}-t + |x-x_{0}|^{2})
\end{split}
\]
as $x\to x_{0}$ and $t \to t_{0}^{-}$.
Since $\mu \geq 0$, it follows that
\begin{align*}
\hspace{1.5cm}\mu(t)u(x,t) &\leq (\mu(t_{0}) + \mu'(t_{0})(t-t_{0}) + o(t_{0}-t))\times\\&\hspace{1cm}\bigg(  u(x_{0},t_{0}) + a(t-t_{0}) + \langle \xi , x-x_{0}\rangle\\&\hspace{1.5cm}+ \frac{1}{2} \langle X(x-x_{0}),x-x_{0} \rangle + o(t_{0}-t + |x-x_{0}|^{2}) \bigg) \\
& = \mu(t_{0})u(x_{0},t_{0}) + (a\mu(t_{0}) + \mu'(t_{0})u(x_{0},t_{0}))(t-t_{0}) \\&\hspace{1cm}+ \mu(t_{0}) \langle \xi ,x-x_{0}\rangle+ \frac{1}{2}\langle \mu(t_{0})X(x-x_{0}),(x-x_{0}) \rangle \\&\hspace{1cm}+ o(|t_0 - t +|x-x_{0}|^{2})
\end{align*}
as $x\to x_{0}$ and $t \to t_{0}^{-}$ and so the claim is true.
\item 
By the definition of $(a,\xi,X) \in \overline{\mathcal{P}}^{2,\pm}u(x_{0},t_{0})$, one can find sequences $(a_{n},\xi_n,X_n)_{n \geq 1} \subseteq \R \times \R^{d} \times \sym_{d}$ and $(x_{n},t_{n})_{n \geq 1} \subseteq \Omega \times (0,T)$  such that $t_{n} \leq t_{0}$ for $n \geq 1$, $(a_{n},\xi_n,X_n) \in \mathcal{P}^{2,\pm}u(x_n,t_{n})$, $(a_{n},\xi_{n},X_{n}) \to (a,\xi,X)$, $(x_{n},t_{n}) \to (x,t)$ and $u(x_{n},t_{n}) \to u(x,t)$ as $n \to \infty$. By part (i),
\[
\hspace{1.5cm}(a_{n} \mu(t_{n}) + \mu'(t_{n})u(x_{n},t_{n}),\mu(t_{n})\xi_{n},\mu(t_{n})X_{n}) \in \mathcal{P}^{2,\pm}(\mu(t_{n})u(x_{n},t_{n})).
\]
In addition to the convergences we have already assumed, since $\mu$ is continuously differentiable, we have $\mu(t_{n}) \to \mu(t_{0})$ and $\mu'(t_{n}) \to \mu'(t_{0})$ as $n \to \infty$. The claim now follows immediately.
\item The statement follows immediately by writing \[u(x_{0},t_{0}) = \frac{1}{\mu(t_{0})} (\mu(t_{0})u(x_{0},t_{0}))\] and applying part (ii).
\end{enumerate}
\end{proof}

\begin{proof}[Proof of Theorem \ref{thm:dne-viscosity}]

We only give a detailed proof in the case $\lambda >0$. Following this, we briefly describe the necessary modifications to adapt the proof to the cases $\lambda =0$ and $\lambda < 0$.

\emph{The $\lambda > 0$ case:}
    
Let $u$ be a positive viscosity solution of \eqref{2025-10-27: 1} and set $v= \frac{m}{\lambda}u^{\lambda}$. Then, as already shown, $v$ is a viscosity solution of \eqref{2025-10-27: 2}.

Fix $\varepsilon > 0$ and define a function $Z$ on $\overline{\Omega} \times \overline{\Omega} \times S$ by
\[
Z(x,y;t,s)=\beta^{\lambda}(t)v(x,t) - \beta^{\lambda}(s)v(y,s) + \omega(x,y;t,s) + \frac{\varepsilon}{2}(t-s)^2 + \varepsilon.
\]
Identically as in the case for classical solutions, we know that $Z>0$ whenever $s$ is sufficiently small, or $s$ is sufficiently close to $t$. In particular, $Z>0$ in a small enough neighbourhood of $(0,0)$. Suppose there is a point $P_{0}:=(x_{0},y_{0};t_{0},s_{0})$ such that $Z=0$ for the first time at $P_0$ in the sense of Definition \ref{def:first-touching-point}. We first assume that $x_{0},y_{0} \in \Omega$.

Applying Corollary \ref{cor:para-max-semi-min} with $u_{1}(x,t)=\beta^{\lambda}(t)v(x,t)$, $u_{2}(y,s)=-\beta^{\lambda}(s)v(y,s)$, and $\varphi= - \omega(x,y;t,s) - \frac{\varepsilon}{2}(t-s)^2 - \varepsilon$ for $(x,y,t,s) \in \Omega \times \Omega \times (s_{0},t_{0}]\times (0,s_{0}]$, we deduce that for every $\kappa>0$, there exist matrices $X,Y \in \sym_{d}$ such that
\begin{align}
(\partial_{t}\varphi(P_{0}),\nabla_{x}\varphi(P_{0}),X) &\in \overline{\mathcal{P}}^{2,-}(\beta^{\lambda}v)(x_{0},t_{0}),\label{2025-10-27: 
3} \\
(\partial_{s}\varphi(P_{0}),\nabla_{y}\varphi(P_{0}),-Y) &\in \overline{\mathcal{P}}^{2,-}(-\beta^{\lambda}v)(y_{0},s_{0}),\label{2025-10-27: 
4}
\end{align}
and
\begin{equation}\label{2025-10-27: 
5}
 \left( \frac{1}{\kappa} + \|A\| \right)I_{2d} \geq \begin{pmatrix}
X  & 0 \\ 
0 & -Y
\end{pmatrix} \geq H - \kappa H^{2},
\end{equation}
where $H= D^{2}_{(x,y)}\varphi(P_{0}) = -D^{2}_{(x,y)}\omega(P_{0})$. Computing the first-order derivatives of $\varphi$ explicitly and applying the product rule in Lemma \ref{lem:product-rule} (iii), the condition \eqref{2025-10-27: 3} implies
\[
\begin{split}
   &\bigg( \frac{- \partial_{t} \omega(P_{0})- \varepsilon(t_{0}-s_{0})}{\beta^{\lambda}(t_{0})} -\lambda v(x_{0},t_{0}) \frac{\beta'(t_{0})}{\beta(t_{0})},\\&\hspace{3cm} - \frac{1}{\beta^{\lambda}(t_{0})}\nabla_{x}\omega(P_{0}),\frac{1}{\beta^{\lambda}(t_{0})}X \bigg) \in \overline{\mathcal{P}}^{2,-}v(x_{0},t_{0}). 
\end{split}
\]
Using that $v$ is a supersolution of \eqref{2025-10-27: 2}, this implies
\[
\begin{split}
    0 &\leq \frac{- \partial_{t} \omega(P_{0})- \varepsilon(t_{0}-s_{0})}{\beta^{\lambda}(t_{0})} -\lambda v(x_{0},t_{0}) \frac{\beta'(t_{0})}{\beta(t_{0})} \\&\hspace{1cm}- \lambda v(x_{0},t_{0})\beta^{-\lambda(p-1)}(t_{0})|\nabla_{x}\omega(P_{0})|^{p-4}\times\\&\hspace{2cm}(|\nabla_{x}\omega(P_{0})|^{2}\trace X + (p-2)\nabla_{x}\omega (P_{0})^{\top}X \nabla_{x}\omega(P_{0})) \\&\hspace{1cm}- \beta^{-\lambda p}(t_{0})|\nabla_{x} \omega(P_{0})|^{p}.
\end{split}
\]
However, since $\beta^{\lambda}(t_{0}) >0$, we may apply \eqref{eq:dne-euclidean-omega-ineq-1-visc} to obtain
\begin{equation}\label{2025-10-27: 
6}
\begin{split}
0 &\leq - \frac{\varepsilon(t_{0}-s_{0})}{\lambda v(x_{0},t_{0})} - \beta^{\lambda-1}(t_{0})\beta'(t_{0}) \\&\hspace{0.5cm}- \beta^{-\lambda(p-2)}(t_{0})|\nabla_{x}\omega(P_{0})|^{p-4}\times \\&\hspace{1.5cm}(|\nabla_{x}\omega(P_{0})|^{2}\trace X + (p-2)\nabla_{x}\omega (P_{0})^{\top}X \nabla_{x}\omega(P_{0})).
\end{split}
\end{equation}
In a similar way, one may derive from \eqref{2025-10-27: 4} using the product rule, that
\[
\begin{split}
    &\bigg( \frac{\partial_{s} \omega(P_{0})- \varepsilon(t_{0}-s_{0})}{\beta^{\lambda}(s_{0})} -\lambda v(y_{0},s_{0}) \frac{\beta'(s_{0})}{\beta(s_{0})}, \\&\hspace{3cm}\frac{1}{\beta^{\lambda}(s_{0})}\nabla_{y}\omega(P_{0}),\frac{1}{\beta^{\lambda}(s_{0})}Y \bigg) \in \overline{\mathcal{P}}^{2,+}v(y_{0},s_{0}).
\end{split}
\]
Since $v$ is also a subsolution of \eqref{2025-10-27: 2}, we have
\[
\begin{split}
   0 &\geq \frac{\partial_{s} \omega(P_{0})- \varepsilon(t_{0}-s_{0})}{\beta^{\lambda}(s_{0})} -\lambda v(y_{0},s_{0}) \frac{\beta'(s_{0})}{\beta(s_{0})} \\&\hspace{1cm}- \lambda v(y_{0},s_{0})\beta^{-\lambda(p-1)}(s_{0})|\nabla_{y}\omega(P_{0})|^{p-4}\times \\&\hspace{2cm}(|\nabla_{y}\omega(P_{0})|^{2}\trace Y + (p-2)\nabla_{y}\omega (P_{0})^{\top}Y \nabla_{y}\omega(P_{0})) \\& \hspace{1cm}- \beta^{-\lambda p}(s_{0})|\nabla_{y} \omega(P_{0})|^{p}, 
\end{split}
\]
which because of \eqref{eq:dne-euclidean-omega-ineq-2-visc} becomes
\begin{equation}\label{2025-10-27: 
7}
\begin{split}
    0 &\geq - \frac{\varepsilon(t_{0}-s_{0})}{\lambda v(y_{0},s_{0})} - \beta^{\lambda-1}(s_{0})\beta'(s_{0}) \\&\hspace{1cm}- \beta^{-\lambda(p-2)}(s_{0})|\nabla_{y}\omega(P_{0})|^{p-4}\times \\&\hspace{2cm}(|\nabla_{y}\omega(P_{0})|^{2}\trace Y + (p-2)\nabla_{y}\omega (P_{0})^{\top}Y \nabla_{y}\omega(P_{0})).
\end{split}
\end{equation}

Taking a linear combination of the inequalities \eqref{2025-10-27: 6} and \eqref{2025-10-27: 7} with the coefficients $A(t_{0})^{2}$ and $A(s_{0})^{2}$ yields
\begin{equation}\label{2025-10-28: 
8}
\begin{split}
    0 &\geq A(t_{0})^{2} \beta^{\lambda}(t_{0})\beta'(t_{0}) - A(s_{0})^{2}\beta^{\lambda}(s_{0})\beta'(s_{0})\\&\hspace{1cm}+ \frac{\varepsilon}{\lambda}(t_{0}-s_{0})\left( \frac{A(t_{0})^{2}}{v(x_{0},t_{0})}-\frac{A(s_{0})^{2}}{v(y_{0},s_{0})} \right) \\&\hspace{1cm}+ A(t_{0})^{2}\beta^{-\lambda(p-2)}(t_{0})|\nabla_{x}\omega(P_{0})|^{p-4}\times\\&\hspace{2cm}(|\nabla_{x}\omega(P_{0})|^{2}\trace X + (p-2)\nabla_{x}\omega (P_{0})^{\top}X \nabla_{x}\omega(P_{0}))\\&\hspace{1cm}- A(s_{0})^{2}\beta^{-\lambda(p-2)}(s_{0})|\nabla_{y}\omega(P_{0})|^{p-4}\times \\&\hspace{2cm}(|\nabla_{y}\omega(P_{0})|^{2}\trace Y + (p-2)\nabla_{y}\omega (P_{0})^{\top}Y \nabla_{y}\omega(P_{0})). 
\end{split}
\end{equation}
We first work to estimate the quantity
\begin{align*}
Q&:= A(t_{0})^{2}\beta^{-\lambda(p-2)}(t_{0})|\nabla_{x}\omega(P_{0})|^{p-4}\times \\&\hspace{2cm}(|\nabla_{x}\omega(P_{0})|^{2}\trace X + (p-2)\nabla_{x}\omega (P_{0})^{\top}X \nabla_{x}\omega(P_{0}))\\&\hspace{1cm}- A(s_{0})^{2}\beta^{-\lambda(p-2)}(s_{0})|\nabla_{y}\omega(P_{0})|^{p-4}\times \\&\hspace{2cm}(|\nabla_{y}\omega(P_{0})|^{2}\trace Y + (p-2)\nabla_{y}\omega (P_{0})^{\top}Y \nabla_{y}\omega(P_{0})) \\
&= \trace \Big(A(t_{0})^{2}\beta^{-\lambda(p-2)}(t_{0})|\nabla_{x}\omega(P_{0})|^{p-2}X - A(s_{0})^{2}\beta^{-\lambda(p-2)}(s_{0})|\nabla_{y}\omega(P_{0})|^{p-2}Y\Big) \\&\hspace{1cm}+ (p-2) \Big(A(t_{0})^{2}\beta^{-\lambda(p-2)}(t_{0})|\nabla_{x}\omega(P_{0})|^{p-4}\nabla_{x}\omega (P_{0})^{\top}X \nabla_{x}\omega(P_{0}) \\&\hspace{2cm}- A(s_{0})^{2}\beta^{-\lambda(p-2)}(s_{0})|\nabla_{y}\omega(P_{0})|^{p-4}\nabla_{y}\omega (P_{0})^{\top}Y \nabla_{y}\omega(P_{0}))\Big) \\
\end{align*}
Setting
\begin{align*}
F&:= \begin{pmatrix}
A(t_{0})^{2}\beta^{-\lambda(p-2)}(t_{0})|\nabla_{x}\omega(P_{0})|^{p-2}I_{d} & 0 \\
0  & A(s_{0})^{2}\beta^{-\lambda(p-2)}(s_{0})|\nabla_{y}\omega(P_{0})|^{p-2}I_{d}
\end{pmatrix}, \\
G&:= \begin{pmatrix}
X & 0 \\
0 & -Y
\end{pmatrix}, \\
\zeta&:= \begin{pmatrix}
A(t_{0}) \beta^{-\frac{\lambda(p-2)}{2}}(t_{0})|\nabla_{x}\omega(P_{0})|^{\frac{p-4}{2}}\nabla_{x}\omega(P_{0}) \\
A(s_{0}) \beta^{-\frac{\lambda(p-2)}{2}}(s_{0})|\nabla_{y}\omega(P_{0})|^{\frac{p-4}{2}}\nabla_{y}\omega(P_{0})
\end{pmatrix},
\end{align*}
we may express $Q$ by
\[
Q=\trace(FG) + (p-2)\zeta^\top G \zeta.
\]
If we set
\[
\tilde{F}:= \begin{pmatrix}
    F_{11} & F_{12} \\ F_{12} & F_{22}
\end{pmatrix},
\]
where
\begin{align*}
    F_{11}&:= A^{2}(t_{0})\beta^{-\lambda(p-2)}(t_{0})|\nabla_x \omega(P_{0})|^{p-2} I_{d},\\
    F_{12}&:= A(t_{0})A(s_{0})\beta^{-\frac{\lambda(p-2)}{2}}(t_{0})\beta^{-\frac{\lambda(p-2)}{2}}(s_{0})|\nabla_x \omega(x_{0},t_{0})|^{\frac{p-2}{2}}|\nabla_y \omega(y_{0},s_{0})|^{\frac{p-2}{2}}I_{d},\\
    F_{22}&:= A^{2}(s_{0})\beta^{-\lambda(p-2)}(s_{0})|\nabla_y \omega(P_{0})|^{p-2} I_{d},
\end{align*}
then it is equivalent to write
\[
Q=\trace(\tilde{F}G) + (p-2)\zeta^\top G \zeta.
\]
According to \eqref{2025-10-27: 5}, we have $G \geq -D^{2}_{(x,y)}\omega(P_{0}) - \kappa(D^{2}_{(x,y)}\omega(P_{0}))^{2}$, and therefore
\[
\begin{split}
    Q &\geq -\trace(\tilde{F}(D^{2}_{(x,y)}\omega(P_{0}) + \kappa(D^{2}_{(x,y)}\omega(P_{0}))^{2})) \\&\hspace{1cm}- (p-2)\zeta^{\top}(D^{2}_{(x,y)}\omega(P_{0}) + \kappa(D^{2}_{(x,y)}\omega(P_{0}))^{2})\zeta
\end{split}
\]
We may now take $\kappa \to 0$ to obtain
\[
Q \geq -\trace(\tilde{F}(D^{2}_{(x,y)}\omega(P_{0}) ) - (p-2)\zeta^{\top}D^{2}_{(x,y)}\omega(P_{0}) \zeta.
\]
Inserting this in \eqref{2025-10-28: 8} with the products written out and using \eqref{eq:dne-euclidean-omega-ineq-3-visc}, we have
\begin{align*}
0 &\geq A(t_{0})^{2} \beta^{\lambda}(t_{0})\beta'(t_{0}) - A(s_{0})^{2}\beta^{\lambda}(s_{0})\beta'(s_{0})\\&\hspace{1cm}+ \frac{\varepsilon}{\lambda}(t_{0}-s_{0})\left( \frac{A(t_{0})^{2}}{v(x_{0},t_{0})}-\frac{A(s_{0})^{2}}{v(y_{0},s_{0})} \right) \\&\hspace{1cm}-A^{2}(t_{0})\beta^{-\lambda(p-2)}(t_{0})\Delta_{p,x}\omega(P_{0}) - A^{2}(s_{0}) \beta^{-\lambda(p-2)}(s_{0})\Delta_{p,y}\omega(P_{0})  \\& \hspace{1cm}-2A(t_{0})A(s_{0})\beta^{-\frac{\lambda(p-2)}{2}}(t_{0})\beta^{-\frac{\lambda(p-2)}{2}}(s_{0})\Delta_{p,xy}\omega(P_{0})  \\
& \geq \frac{\varepsilon}{\lambda}(t_{0}-s_{0})\left( \frac{A(t_{0})^{2}}{v(x_{0},t_{0})}-\frac{A(s_{0})^{2}}{v(y_{0},s_{0})} \right)>0,
\end{align*}
which is a contradiction, as before. Therefore, if such a point $P_{0}$ exists, it must occur with at least one of $x_{0}$ or $y_{0}$ on the boundary $\partial \Omega$.

We now make no assumption regarding whether $x_{0},y_{0} \in \Omega$ or $x_{0},y_{0} \in \partial \Omega$. Proceeding in the same manner as in the proof of Theorem \ref{thm:Euclidean-Schroedinger-full-visc}, we choose a point $z_{0} \in \Omega$ with $\delta:=d(z_{0},\partial \Omega) >0$ and define
\[
h(x):= \frac{1}{2} |x-z_{0}|^{2}
\]
for all $x \in \overline{\Omega}$. Letting
\[
v^{\theta}(x,t):= v(x,t) + \theta h(x), \quad \text{and} \quad v_{\theta}(x,t):=v(x,t) - \theta h(x)
\]
for $(x,t) \in \overline{\Omega} \times (0,T)$ and $\theta>0$ small, we have that $v^{\theta}$ satisfies
\begin{equation}\label{2025-11-14: 
v-upper-prob-dne}
 \begin{cases}
   \begin{alignedat}{2}
       \partial_t v^{\theta} - \lambda( v^{\theta} - \theta h)\Delta_{p} (v^{\theta} - \theta h) - |\nabla(v^{\theta} - \theta h)|^{p} &= 0&& \qquad \text{in $\Omega \times (0,T)$,} \\
        |\nabla v^{\theta} -\theta _{n} \nabla h|^{p-2} (\nabla v^{\theta} -\theta _{n}\nabla h)\cdot \nu(x_{\theta _{n}}) &= 0 && \qquad \text{on $\partial \Omega \times (0,T)$,}
   \end{alignedat}
\end{cases}
\end{equation}
and $v_{\theta}$ satisfies
\begin{equation}\label{2025-11-14: 
v-lower-prob-dne}
 \begin{cases}
   \begin{alignedat}{2}
       \partial_t v^{\theta} - \lambda( v^{\theta} + \theta h)\Delta_{p} (v^{\theta} + \theta h) - |\nabla(v^{\theta} + \theta h)|^{p} &= 0&& \qquad \text{in $\Omega \times (0,T)$,} \\
        |\nabla v^{\theta} +\theta _{n} \nabla h|^{p-2} (\nabla v^{\theta} +\theta _{n}\nabla h)\cdot \nu(x_{\theta _{n}}) &= 0 && \qquad \text{on $\partial \Omega \times (0,T)$,}
   \end{alignedat}
\end{cases}
\end{equation}
both in the sense of viscosity solutions.

We again modify the function $Z$ by defining
\[
\tilde{Z}_{\theta}(x,y;t,s):= \beta^{\lambda}(t)v^{\theta}(x,t) - \beta^{\lambda}(s)v_{\theta}(y,s)  - \varphi(x,y;t,s),
\]
where
\[
\begin{split}
  \varphi(x,y,;t,s)&:= - \omega(x,y;t,s) - \frac{\varepsilon}{2}(t-s)^{2}- \varepsilon \\& \hspace{1cm}- |x-x_{0}|^{4} - |y-y_{0}|^{4} - |t-t_{0}|^{4} - |s-s_{0}|^{4}  
\end{split}
\]
for all $(x,y;t,s) \in \tilde{K}:= \overline{\Omega} \times \overline{\Omega} \times (t_{0}-\eta,t_{0}] \times (s_{0} - \eta , s_{0}]$ and $\theta>0$, where $0 < \eta < t_{0}-s_{0}$. By the same argument as before, we can find a sequence of points $(P_{\theta_{n}})_{n \geq 1} \subseteq \tilde{K}$, such that $P_{\theta_{n}}$ minimises $\tilde{Z_{\theta _{n}}}$ over $\tilde{K}$ for each $n \geq 1$ and $P_{\theta _{n}} \to P_{0}$ in $\tilde{K}$. Applying Corollary \ref{cor:para-max-semi-min} to $\tilde{Z}_{\theta _{n}}$ at $P_{\theta _{n}}$, we find that for all $\kappa>0$, there exist $X_n,Y_n \in \sym_d$ such that
\begin{align*}
(\partial_{t}\varphi(P_{\theta _{n}}),\nabla_{x}\varphi(P_{\theta _{n}}),X_{n}) &\in \overline{\mathcal{P}}^{2,-}_{\tilde{K}}(\beta^{\lambda}v^{\theta_{n}})(x_{\theta _{n}},t_{\theta _{n}}), \\
(-\partial_{s}\varphi(P_{\theta _{n}}),-\nabla_{y}\varphi(P_{\theta _{n}}),Y_{n}) &\in \overline{\mathcal{P}}^{2,+}_{\tilde{K}}(\beta^{\lambda}v_{\theta _{n}})(y_{\theta _{n}},s_{\theta _{n}}),
\end{align*}
and also 
\begin{equation}\label{eq:matrix-ineq-bdry-dne}
 \left( \frac{1}{\kappa} + \|H_{n}\| \right)I_{2d} \geq \begin{pmatrix}
X_{n}  & 0 \\ 
0 & -Y_{n}
\end{pmatrix} \geq H_{n} - \kappa H_{n}^{2},
\end{equation}
where $H_{n}:= D^{2}\varphi(P_{\theta _{n}})$. Using the product rule and evaluating the derivatives of $\varphi$, this implies
\begin{align*}
(a^{n},\xi^{n}, \Upsilon^{n})&:= \bigg(  \frac{- \partial_{t} \omega (P_{\theta _{n}}) - 4 |t_{\theta _{n}} - t_{0}|^{2}(t_{\theta _{n}} - t_{0}) - \varepsilon (t_{\theta _{n}} - s_{\theta _{n}})}{\beta^{\lambda} (t_{\theta _{n}})} \\& \hspace{1cm}- \lambda v^{\theta _{n}}(x_{\theta _{n}},t_{\theta _{n}}) \frac{\beta'(t_{\theta _{n}})}{\beta(t_{\theta _{n}})}, \\& \hspace{1cm}-\frac{1}{\beta^{\lambda}(t_{\theta _{n}})} (\nabla_{x} \omega (P_{\theta _{n}}) + 4|x_{\theta _{n}}-x_{0}|^{2}(x_{\theta _{n}}-x_{0})), \frac{1}{\beta^{\lambda}(t_{\theta _{n}})} X_{n} \bigg) \\& \hspace{7.5cm}\in \overline{\mathcal{P}}^{2,-}_{\tilde{K}}(v^{\theta_{n}})(x_{\theta _{n}},t_{\theta _{n}}), \\
(a_{n},\xi_{n}, \Upsilon_{n})&:= \bigg(  \frac{ \partial_{s} \omega (P_{\theta _{n}}) + 4 |s_{\theta _{n}} - s_{0}|^{2}(s_{\theta _{n}} - s_{0}) - \varepsilon (t_{\theta _{n}} - s_{\theta _{n}})}{\beta^{\lambda} (s_{\theta _{n}})} \\& \hspace{1cm} - \lambda v_{\theta _{n}}(y_{\theta _{n}},s_{\theta _{n}}) \frac{\beta'(s_{\theta _{n}})}{\beta(s_{\theta _{n}})}, \\& \hspace{1cm}\frac{1}{\beta^{\lambda}(s_{\theta _{n}})} (\nabla_{y} \omega (P_{\theta _{n}}) + 4|y_{\theta _{n}}-y_{0}|^{2}(y_{\theta _{n}}-y_{0})), \frac{1}{\beta^{\lambda}(s_{\theta _{n}})} Y_{n} \bigg) \\& \hspace{7.5cm}\in \overline{\mathcal{P}}^{2,+}_{\tilde{K}}(v_{\theta_{n}})(y_{\theta _{n}},s_{\theta _{n}}).
\end{align*}
Since $v^{\theta_{n}}$ is a viscosity subsolution of \eqref{2025-11-14: v-upper-prob-dne}, we have that
\begin{equation}\label{2025-11-19: 
positive thing one}
\begin{split}
     Q_{n}&:=  a^{n} - \lambda (v^{\theta _{n}}(x_{\theta _{n}},t_{\theta _{n}}) - \theta _{n}\nabla h(x_{\theta _{n}}))|\xi^{n}-\theta _{n}\nabla h(x_{\theta _{n}})|^{p-4}\times \\&\hspace{1cm}\Big(|\xi^{n}-\theta _{n}\nabla h(x_{\theta _{n}})|^{2}\trace (\Upsilon^{n}-\theta _{n}D^2 h(x_{\theta _{n}})) \\&\hspace{1.5cm}+ (p-2)(\xi^{n}-\theta _{n}\nabla h(x_{\theta _{n}}))^{\top}(\Upsilon^{n}-\theta _{n}D^2 h(x_{\theta _{n}})) (\xi^{n}-\theta _{n}\nabla h(x_{\theta _{n}}))\Big) \\&\hspace{1cm}- |\xi^{n}-\theta _{n}\nabla h(x_{\theta _{n}})|^{p}\geq 0
\end{split}
\end{equation}
if $x_{\theta _{n}} \in \Omega$ and
\[
\max \{ Q_{n}, |\xi^{n} -\theta _{n} \nabla h(x_{\theta _{n}})|^{p-2} (\xi^{n} -\theta _{n}\nabla h(x_{\theta _{n}}))\cdot \nu(x_{\theta _{n}})  \} \geq 0
\]
if $x_{\theta _{n}} \in \partial \Omega$. By the same argument as in the proof of Theorem \ref{thm:Euclidean-Schroedinger-full-visc}, we conclude that
\[
\begin{split}
    (\xi^{n} &-\theta _{n}\nabla h(x_{\theta _{n}}))\cdot \nu(x_{\theta _{n}}) \\&= \left(-\frac{1}{\beta^{\lambda}(t_{\theta _{n}})} (\nabla_{x} \omega (P_{\theta _{n}}) + 4|x_{\theta _{n}}-x_{0}|^{2}(x_{\theta _{n}}-x_{0})) - \theta _{n} \nabla h(x_{\theta _{n}})\right)\cdot \nu(x_{\theta _{n}}) \\&< -\theta _{n} \delta <0,
\end{split}
\]
and in particular $\xi^{n} -\theta _{n}\nabla h(x_{\theta _{n}}) \ne 0$. Therefore, we have
\[
|\xi^{n} -\theta _{n} \nabla h(x_{\theta _{n}})|^{p-2} (\xi^{n} -\theta _{n}\nabla h(x_{\theta _{n}}))\cdot \nu(x_{\theta _{n}}) <0,
\]
which implies \eqref{2025-11-19: positive thing one} holds whether $x_{\theta _{n}} \in \Omega$ or $x_{\theta _{n}} \in \partial\Omega$. By the same argument, we can also show that
\begin{equation}\label{2025-11-19: 
positive thing two}
\begin{split}
    & a_{n} - \lambda (v_{\theta _{n}}(y_{\theta _{n}},s_{\theta _{n}}) - \theta _{n}\nabla h(y_{\theta _{n}}))|\xi_{n}-\theta _{n}\nabla h(y_{\theta _{n}})|^{p-4}\times \\& \hspace{1cm}\Big(|\xi _{n}-\theta _{n}\nabla h(y_{\theta _{n}})|^{2}\trace (\Upsilon_{n}-\theta _{n}D^2 h(y_{\theta _{n}})) \\& \hspace{1.5cm}+ (p-2)(\xi_{n}-\theta _{n}\nabla h(y_{\theta _{n}}))^{\top}(\Upsilon_{n}-\theta _{n}D^2 h(y_{\theta _{n}})) (\xi_{n}-\theta _{n}\nabla h(y_{\theta _{n}}))\Big) \\& \hspace{1cm}- |\xi_{n}-\theta _{n}\nabla h(y_{\theta _{n}})|^{p}\geq 0.
\end{split}
\end{equation}
Taking $n\to \infty$ in \eqref{eq:matrix-ineq-bdry-dne}, \eqref{2025-11-19: positive thing one} and \eqref{2025-11-19: positive thing two} yields that

\begin{align*}
    &\frac{- \partial_{t} \omega(P_{0})- \varepsilon(t_{0}-s_{0})}{\beta^{\lambda}(t_{0})} -\lambda v(x_{0},t_{0}) \frac{\beta'(t_{0})}{\beta(t_{0})} \\&\hspace{1cm}- \lambda v(x_{0},t_{0})\beta^{-\lambda(p-1)}(t_{0})|\nabla_{x}\omega(P_{0})|^{p-4}\times\\&\hspace{2cm}(|\nabla_{x}\omega(P_{0})|^{2}\trace X + (p-2)\nabla_{x}\omega (P_{0})^{\top}X \nabla_{x}\omega(P_{0})) \\&\hspace{1cm}- \beta^{-\lambda p}(t_{0})|\nabla_{x} \omega(P_{0})|^{p} \geq 0,\\[18pt]
&\frac{\partial_{s} \omega(P_{0})- \varepsilon(t_{0}-s_{0})}{\beta^{\lambda}(s_{0})} -\lambda v(y_{0},s_{0}) \frac{\beta'(s_{0})}{\beta(s_{0})} \\&\hspace{1cm}- \lambda v(y_{0},s_{0})\beta^{-\lambda(p-1)}(s_{0})|\nabla_{y}\omega(P_{0})|^{p-4}\times \\&\hspace{2cm}(|\nabla_{y}\omega(P_{0})|^{2}\trace Y + (p-2)\nabla_{y}\omega (P_{0})^{\top}Y \nabla_{y}\omega(P_{0})) \\& \hspace{1cm}- \beta^{-\lambda p}(s_{0})|\nabla_{y} \omega(P_{0})|^{p} \geq 0, \\[18pt]
&\left( \frac{1}{\kappa} + \|H\| \right)I_{2d} \geq \begin{pmatrix}
X  & 0 \\ 
0 & -Y
\end{pmatrix} \geq H - \kappa H^{2}\\
\end{align*}
hold as before and we arrive to a contradiction from here.

\emph{The $\lambda =0$ case}: The proof closely follows the procedure used to prove the linear case in Theorem \ref{thm:Euclidean-Schroedinger-full-visc}, repeating the modifications used in the $\lambda > 0$ to accommodate the nonlinearity.

\emph{The $\lambda < 0$ case}: The proof is carried out identically as in the case $\lambda > 0$, except one replaces the $\varepsilon(t-s)^2$ term occurring in the definition of $Z$ by \[
\frac{\varepsilon}{2} \left( \int_{s}^{t} \frac{1}{A^{2}(\tau)\beta^{\lambda}(\tau)} \dd \tau \right)^{-2}.
\]
\end{proof}

\begin{remark}
We identify the function 
\begin{equation}\label{eq:nonlinear-omega}
\omega(x,y;t,s)=\frac{1}{p'}\left( \frac{K}{d} \right) ^{p'-1} \frac{|x-y|^{p'}}{(t^{Kp/d} - s^{Kp/d})^{p'-1}}
\end{equation} 
to be an optimal candidate for the function $\omega$.
One sees that the function $\omega$ defined in \eqref{eq:nonlinear-omega} satisfies the inequalities \eqref{eq:dne-euclidean-omega-ineq-1-visc} and \eqref{eq:dne-euclidean-omega-ineq-2-visc} with equality. In addition, \eqref{eq:dne-euclidean-omega-ineq-3-visc} is also satisfied with equality once one chooses $A(t) = \frac{d}{Kp}t^{(\frac{p}{d} + \frac{\lambda(p-2)}{2})K}$ and $\beta(t)=t^K$ with $K=\frac{1}{m(p-1)-1+\frac{p}{d}}$. Our choice is further supported by the fact that the Barenblatt solution $\mathcal{B}$ for the doubly nonlinear heat equation satisfies the resulting Harnack inequality \eqref{eq:harnack-dne-lam>0-visc} with equality for all $x,y \in \R^d$ and $0<s<t$ in the support of $\mathcal{B}$ satisfying \[
s^{Kp/d}x = t^{Kp/d}y.
\]
\end{remark}

\begin{remark}
    Theorem \ref{thm:dne-viscosity} requires that the function $\beta$ is strictly increasing, which if $\beta(t) = t^K$, only occurs if $K=\frac{1}{m(p-1)-1+\frac{p}{d}}>0$. This introduces the restriction \[
    p> \frac{d(m+1)}{1+dm},
    \]
    which must be satisfied in order to deduce a Harnack inequality using Theorem \ref{thm:dne-viscosity}. In the case $m=1$ corresponding to the $p$-heat equation, this restriction reads as \[
    p> p_{*}:=\frac{2d}{d+1},
    \]
    which matches the range for which both the Harnack inequality (see, for instance, \cite{DiBenedetto-Harnack-Book}) and its differentiated form due to Esteban and Vázquez \cite{Esteban-Vasquez-1990} are known to hold. Moreover, with $\omega$ chosen as in \eqref{eq:nonlinear-omega}, we recover the Harnack inequality
    \[
    \frac{1}{\lambda}(t^{K}u(x,t))^{\lambda} - \frac{1}{\lambda}(s^{K}u(y,s))^{\lambda} + \frac{1}{p'}\left( \frac{K}{d}\right)^{p'-1} \frac{|x-y|^{p'}}{(t^{Kp/d} - s^{Kp/d})^{p'-1}} \geq 0
    \]
    satisfied by every positive solution of the $p$-heat equation for every $x,y \in \R^d$, $0<s<t$, which coincides with the result in \cite{Auchmuty-Bao-1994}.
    
    Similarly, for $p=2$, one has the restriction \[
    m> 1 - \frac{2}{d},
    \]
    which already appeared in the work of Aronson and B\'enilan \cite{Aronson-Benilan-1979} on the porous medium equation. In addition, we also obtain the Harnack inequality 
    \[
    \frac{m}{m-1}(t^{K}u(x,t))^{m-1} - \frac{m}{m-1}(s^{K}u(y,s))^{m-1} +  \frac{K}{2d} \frac{|x-y|^{2}}{t^{2K/d} - s^{2K/d}} \geq 0
    \]
     satisfied by every positive solution of the porous medium equation for every $x,y \in \R^d$, $0<s<t$ \cite{Auchmuty-Bao-1994}.
\end{remark}

\appendix
\section{Approximation argument}
In this brief appendix, we explain why it is possible to approximate weak solutions of the doubly nonlinear heat equation \eqref{eq:dne-full-space} on the full space $\R^d$ by viscosity solutions of the homogeneous Neumann problem \eqref{eq:dne-euclidean-neumann-visc} on bounded domains $\Omega \subseteq \R^d$, which is the crucial fact enabling us to obtain Theorem \ref{thm:dne-full-euclidean}. We first clarify what is meant by a weak solution of homogeneous Neumann problem for the doubly nonlinear heat equation on a bounded domain $\Omega \subseteq \R^d$.

\begin{definition}[Weak solution of \eqref{eq:dne-euclidean-neumann-visc}]
    Let $\Omega \subseteq \R^d$, $T>0$, $m>0$, and  $p>1$. We call a function $u:\Omega \times [0,T] \to [0,\infty)$ a \emph{(nonnegative) weak solution of \eqref{eq:dne-euclidean-neumann-visc}} if \[
    u \in C([0,T]; L^{m+1}(\Omega)), \quad u^m \in L^{p}(0,T;W^{1,p}(\Omega)),
    \]
    and \[
    \int_{0}^{T} \int_{\Omega} -u \partial_{t}\varphi + |\nabla u^{m}|^{p-2}\nabla u^{m}\nabla \varphi \dd x  \dd t =0 
    \]
    for all nonnegative functions \[
    \varphi \in W^{1, 1+ \frac{1}{m}}(0,T;L^{1+ \frac{1}{m}}(\Omega)) \cap L^p(0,T;W^{1,p}_0(\Omega)).
    \]
\end{definition}

By replacing $\Omega$ by $\R^d$ in the above definition, we also understand the meaning of a weak solution of \eqref{eq:dne-manifold} on the full space $\R^d$. 

In the argument to approximate a weak solution of \eqref{eq:dne-manifold} on $\R^d$ by viscosity solutions on bounded sets, one passes to the notion of mild solutions. For this, one must first reinterpret the problem \eqref{eq:dne-manifold} as an abstract Cauchy problem
\begin{equation}\label{eq:dne-cauchy}
   \begin{cases}
   \begin{alignedat}{2}
       \frac{du}{dt}(t) + A u&\ni 0 && \qquad \text{in $(0,T)$,} \\
        u(0) &= u_{0}
   \end{alignedat}
\end{cases} 
\end{equation}
governed by a suitable operator $A$. For the reader's convenience, we give the details of the construction of the operator $A$. We proceed as in \cite{collier-phd} and \cite{collier-hauer-arxiv}. We define the energy
\[
E(u):= \begin{cases}
\frac{1}{p}\int_{\Omega} |\nabla u|^{p}  \dd x  & \text{if } u \in W^{1,p}(\Omega ) \cap L^2(\Omega) \\ + \infty & \text{otherwise}
\end{cases}
\]
for each $u \in L^2(\Omega)$ and define the (negative) Neumann $p$-Laplace operator by \[
-\Delta_p u:= \partial_{L^2(\Omega)}E(u)
\]
for every $u \in D(\partial_{L^2(\Omega)}E)$. Next, we take the part of the operator $-\Delta_p$ in $(L^1(\Omega) \times L^{\infty}(\Omega))^2$, that is
\[
(-\Delta_p)_{\Omega,1 \cap \infty}:= \{ (u,v) \in (L^1(\Omega) \cap L^{\infty}(\Omega))^{2} \mid v = (-\Delta_p)_{\Omega}u \}
\]
Finally, we define the doubly nonlinear operator $A=(-\Delta_p) \circ \cdot ^m$ by
\[
A:= \left\{ (u,v) \in (L^1(\Omega))^{2} \ \left| \  \begin{aligned}\exists ((u_{k},v_{k}))_{k \geq 1} \text{ s.t. } v_{k} = (-\Delta_p)_{\Omega,1 \cap \infty}u^{m}, \\ u_{k}\to u \text{ and } v_{k} \to v \text{ in } L^1(\Omega)\end{aligned}\right.\right\}.
\]

\begin{definition}
    [Mild solution of \eqref{eq:dne-cauchy}]
    A function $u \in C([0,\infty];L^1(\Omega))$ is called a mild solution of the Cauchy problem \eqref{eq:dne-cauchy} if $u(0)=u_{0}$ in $L^1(\Omega)$, and for every $\varepsilon >0$, there is a partition $\pi_\varepsilon:0=t_0<t_1<\ldots < t_{N_{\varepsilon}}=T$ such that $\max_{1, \ldots, N_{\varepsilon}} (t_i - t_{i-1})<\varepsilon$, and there is a sequence $(u_{i, \varepsilon})_{i=0}^{N_\varepsilon} \subseteq D(A)$ defined by $u_{0,\varepsilon}=u_0$ and $u_{i,\varepsilon}$ is defined via
    \[
    u_{i,\varepsilon} + (t_{i,\varepsilon} - t_{i-1,\varepsilon})Au_{i,\varepsilon} \ni u_{i-1,\varepsilon}
    \]
    for $i=1, \ldots, N_\varepsilon - 1$, such that the sequence $(u_{\varepsilon})_{\varepsilon >0}$ of step functions \[
    u_{\varepsilon}(t):= \sum_{i=0}^{N_{\varepsilon}-1} u_{i,\varepsilon} \mathbbm{1}_{[t_{i,\varepsilon},t_{i,\varepsilon+1})}
    \]
    converges to $u$ in $C([0,\infty];L^1(\Omega))$ as $\varepsilon \to 0^{+}$.
\end{definition}

We have the following theorem.

\begin{theorem}\label{thm:approximation}
Let $u_0 \in L^1(\R^d)$ be positive and let $u$ be the unique weak solution to \eqref{eq:dne-manifold} on $\R^d$. Then there exists a sequence $(u_n)_{n \geq 1}$ of functions $u_n: \Omega_n \to (0,\infty)$ defined on an increasing sequence of bounded convex sets $\Omega_n \subseteq \R^d$ such that $u_n$ is a positive viscosity solution of the homogeneous Neumann problem \eqref{eq:dne-euclidean-neumann-visc} for all $n \geq 1$ and $u_n \to u$ pointwise on $\R^d \times (0,\infty)$.
\end{theorem}

The main idea to prove Theorem \ref{thm:approximation} is as follows. We fix a positive initial datum $u_0 \in L^1(\R^d)$ and denote by $u$ the unique weak solution of the doubly nonlinear heat equation \eqref{eq:dne-manifold} on $\R^d$. Then, setting $\Omega_n:= B(0,n)$, we note that for every $n\geq 1$, there exists a unique weak solution $u_n$ to the homogeneous Neumann problem \eqref{eq:dne-euclidean-neumann-visc} with initial datum $u_{0,n}:= u_0|_{\Omega_n}$. In particular, each of the functions $u_n$ is also a viscosity solution of \eqref{eq:dne-manifold} due to a result by Bögelein and Strunk \cite{boegelein-strunk}. 

Next, one demonstrates the convergence of $u_n$ to $u$. We omit, however, the details of the approximation argument itself and instead refer the reader to \cite[Proposition 3.14]{collier-hauer-arxiv}, where an analogous procedure is used in a nonlocal setting. 
Following these ideas, one that shows $u_n$ converges uniformly on $[0,T]$ to a mild solution $\tilde{u}$ of \eqref{eq:dne-cauchy} on $\Omega=\R^d$ for every $T>0$. However, due to regularisation effects of the operator $A$ (see \cite{coulhon-hauer}), this solution $\tilde{u}$ belongs to $L^{\infty}([\varepsilon,\infty); L^{\infty}(\R^d))$ for all $\varepsilon > 0$ and is in fact a weak solution, which we identify as $u$. Finally, using local regularity results, such as Hölder estimates (see, for instance \cite{dibenedetto-degenerate} or \cite{porzio-vespri}), one finds that $u_n \to u$ locally uniformly in $\R^d \times [0,\infty)$.

\section*{Statements and Declarations}
\textbf{Data availability statement.} As this work is mathematical in nature, there are no associated datasets and all relevant information is available in this manuscript.\\

\textbf{Funding.} The first author's research was supported by the Laureate Grant FL150100126 and Discovery Projects Grant DP250103808 of the Australian Research Council. The second author's research was supported by the Australian Research Council grant DP220100067. The third author's research was supported by an Australian Government Research Training Program (RTP) Scholarship. The authors are grateful for the support.\\

\textbf{Conflict of Interest Statement.} The authors have no conflicts of interest to declare.

\bibliographystyle{abbrvurl}

\end{document}